\documentclass[a4paper,10pt,DIV=15,
twocolumn
]{scrartcl}
\usepackage[utf8]{inputenc}
\usepackage{fontenc}
\usepackage[english]{babel}
\usepackage{amsmath, amsthm, amssymb,mathtools}
\mathtoolsset{showonlyrefs}
\usepackage{graphicx}
\usepackage{subcaption}
\usepackage{xargs}
\usepackage{algorithm}
\usepackage[noend]{algpseudocode}
\usepackage{booktabs}
\usepackage{enumitem}
\usepackage{listings}
\usepackage{environ,ifthen,xcolor}
\usepackage{flushend}
\usepackage{hyperref}
\newtheorem{theorem}{Theorem}[section]

\newtheorem{proposition}[theorem]{Proposition}
\newtheorem{lemma}[theorem]{Lemma}
\newtheorem*{remark-Non}{Remark}

\newtheorem{example}[theorem]{Example}

\definecolor{tuklblue}{RGB}{0,95,140}
\usepackage[
  backend=bibtex,
  style=numeric-comp,
  firstinits=true,
  natbib=true,
  maxcitenames=5,
  sortlocale=en_EN,
  url=false, 
  doi=true,
  eprint=false
]{biblatex}
\setkomafont{section}{\large}
\setkomafont{subsection}{\normalsize}
\setkomafont{sectioning}{\rmfamily\bfseries}
\setkomafont{title}{\rmfamily}
\DeclareCaptionLabelSeparator{periodspace}{.\ }
\newcommand{\R}{\mathbb{R}}

\newcommand{\Z}{\mathbb{Z}}

\newcommand{\tT}{\mathrm{T}}
\makeatletter
\newcommandx{\abs}[2][1=\@empty]{#1\lvert #2 #1\rvert}
\newcommandx{\norm}[3][1=\@empty,3=\@empty]{#1\lVert #2 #1\rVert_{#3}}
\makeatletter

\DeclareMathOperator*{\argmin}{arg\,min}
\DeclareMathOperator{\prox}{prox}

\DeclareMathOperator{\sgn}{sgn}

\hypersetup{pdfauthor={Ronny Bergmann, Andreas Weinmann},
	pdftitle={A Second Order TV-type Approach for 
Inpainting and Denoising Higher Dimensional Combined Cyclic and Vector Space Data},
	pdfsubject={revised manuscript},%
	pdfcreator = {pdflatex and TextMate},%
	pdfkeywords={},
	plainpages=false, pdfstartview=FitH, pdfview=FitH, pdfpagemode=UseOutlines,%
	bookmarksnumbered=true,bookmarksopen=false,bookmarksopenlevel=0,%
	colorlinks=true,linkcolor=black,citecolor=black,urlcolor=black%
}
\title{A Second Order TV-type Approach for 
Inpainting and Denoising Higher Dimensional Combined Cyclic and Vector Space Data}
\author{Ronny Bergmann\thanks{Departement of Mathematics,
Technische Universit\"at Kaiserslautern,
Paul-Ehrlich-Str. 31, 67663 Kaiserslautern, Germany
\newline bergmann@mathematik.uni-kl.de.}
\and
Andreas Weinmann\thanks{Department of Mathematics,
Technische Universit\"at M\"unchen and
Fast Algorithms for Biomedical Imaging Group, Helmholtz-Zentrum M\"unchen,
Ingolst\"adter Landstr. 1, 85764 Neuherberg, Germany
\newline andreas.weinmann@tum.de}%
}
\date{\today}
\begin{document}
     \maketitle 
     \begin{abstract}\textbf{Abstract.}
	In this paper we consider denoising and inpainting problems for higher
	dimensional combined cyclic and linear space valued data.
	These kind of data appear when dealing with nonlinear color spaces such as HSV, and
	they can be obtained by changing the space domain of, e.g., an optical flow
	field to polar coordinates.
	For such nonlinear data spaces, we develop algorithms for the solution of the
	corresponding second order total variation (TV) type problems for denoising,
	inpainting as well as the combination of both. We provide a convergence
	analysis and we apply the algorithms to concrete problems.
     \end{abstract}
\section{Introduction}
\label{sec:Introduction}
%
%
One of the most well known methods for edge-preserving image denoising is the
variational approach minimizing the Rudin-Osher-Fatemi (ROF)
functional~\cite{ROF92}.
In its basic form, it deals with scalar data.
Related variational approaches for vector space valued data have gained a lot
of interest in the literature and are still topic of ongoing research;
we exemplarily refer
to~\cite{BC98,raket2011tv,goldluecke2012natural,MSMC14:LowRank} and the
references therein. In this paper, we consider TV-type functionals
incorporating first and second order differences
for the nonlinear data spaces which combine vector space valued data ---in the
following called linear space data to avoid confusion--- and vectors of cyclic
data. In these spaces, we deal with denoising and inpainting problems as well
as simultaneous inpainting and denoising problems.

Image inpainting is a problem arising in many applications in image processing,
image analysis and related fields.
Examples are restoring scratches in photographs, removal of superimposed
objects, dealing with an area removed by a user, digital zooming as well as
edge decoding.
Principally, any missing data situation ---whatever the reason might be---
results in an inpainting problem.
This is not restricted to 2D images.
Further examples are defects in audio and video recordings, or in seismic data
processing.
In this respect, also interpolation, approximation, and extrapolation problems
may be viewed as inpainting problems.
We recommend the survey~\cite{guillemot2014image} and Chapter 6
of~\cite{chan2005image} as well as~\cite{BBCS10,CDOS12} for an overview on
inpainting and for further applications. 
There are various conceptionally different approaches to inpainting,
cf.~\cite{chan2005image,guillemot2014image} and~\cite{CDOS12} which also
includes some comparison.
Among these are methods based on linear transforms from harmonic analysis such
as curvelets and shearlets which are combined with a sparsity approach
based on $\ell^1$ minimization on the corresponding coefficients%
~\cite{cai2010simultaneous,elad2005simultaneous,dong2012wavelet,king2014analysis}.
Other approaches are based on (often nonlinear) PDE and variational models,
cf., e.g.,%
~\cite{bornemann2007fast,caselles1998axiomatic,chan2006error,chan2002euler,esedoglu2002digital,marz2013well,marz2011image,masnou2002disocclusion,masnou1998level,shen2002mathematical,tschumperle2006fast}.
\\In general, exemplar-based and sparsity-based methods perform better for
filling large texture areas,
whereas diffusion-based and variational techniques yield better results for
natural images.
Among the variational techniques applied, total variation (TV) minimization is
one of the prominent models. The minimizer of the corresponding TV functional yields the inpainted
image.
TV inpainting works well for elongated inpainting areas but has problems with
larger gap connections. In such situations higher order, in particular curvature based, schemes perform better.

The first TV regularized model was proposed in~\cite{ROF92} for denoising. 
It was first applied to missing data situations/inpainting in~\cite{BBCSV01,CS01}.
Further references for TV based image inpainting  are~\cite{bresson2008fast,chan2005simultaneous,shen2002mathematical}.
In contrast to classical methods, the results are typically not over-smoothed;  
however, it is well known that these minimizers very often show `staircasing'
effects, i.e.\ the result is often piecewise constant, although the underlying
signal varies smoothly in the corresponding regions.
In order to avoid staircasing, higher order and, in particular, 
second order differences and derivatives (in a continuous domain setting),
are often employed. References are the pioneering work~\cite{CL97} as well as%
~\cite{BKP09,CEP07,CMM00,DSS09,DWB09,HS06,LBU2012,LLT03,LT06,Sche98,SS08,SST11}.
TV functionals for linear space valued data were considered in \cite{BC98} 
in the context of linear color spaces. 
The papers \cite{PS14,PSS13} deal with denoising and inpainting in the RGB color space
using linear combinations of first and second order terms.
A total generalized variational model can be found, e.g., in~\cite{MSMC14:LowRank}. 
In contrast to applying pure second order terms or linear combinations of first and second order
terms, total generalized variational approaches try to find some optimal balancing between the first and second derivatives.   
The advantage is that the related schemes better preserve the edge structures.
A detailed description may be found in \cite{BKP09}.
The authors of \cite{MSMC14:LowRank} obtain a model for
denoising linear space valued color data.
Second order total generalized variation was generalized for tensor fields in \cite{VBK13}.

However, in many applications,
data having values in nonlinear spaces appear.
Examples are diffusion tensor images~\cite{basser1994mr,pennec2006riemannian},
color images based on non-flat color models%
~\cite{chan2001total,kimmel2002orientation,lai2014splitting,vese2002numerical}
or motion group-valued data \cite{rosman2012group,rahman2005multiscale}.
Due to its importance, processing such manifold valued data has gained a lot of
interest in recent years.
To mention only some examples, wavelet-type multi scale transforms
for manifold data have been considered in%
~\cite{GW09,rahman2005multiscale,weinmann2012interpolatory}.
Statistical issues on Riemannian manifolds are the topic of%
~\cite{bhattacharya2003large,bhattacharya2005large,Fle13,oller1995intrinsic,Pen06}
and circular data are, in particular, considered in~\cite{fisher95,JS2001}.
Furthermore, manifold-valued partial differential equations are studied in~\cite{chefd2004regularizing,GHS13,tschumperle2001diffusion}.

For TV functionals for manifold-valued data, an analysis  
from a theoretical viewpoint has been carried out in~\cite{GM06,GM07}.
These papers extend previous work~\cite{GMS93} on $\mathbb S^1$-valued 
functions where, in particular, the existence of minimizers of certain TV-type energies is shown. 
An algorithm for TV minimization on Riemannian manifolds was proposed in%
~\cite{LSKC13}.
This approach uses a reformulation as a multi label optimization problem with an
infinite number of labels and a subsequent convex relaxation.
An approach for linear spaces using a relaxation of the label optimization
problem as well was presented in~\cite{GSC13}.
First order TV minimization for $\mathbb S^1$-valued data has been considered
in~\cite{SC11,CS13}.
In particular, these authors consider inpainting for $\mathbb S^1$-valued data in a
first order TV setup.
We proposed a different approach to first order TV minimization for manifold-valued data via
cyclic and parallel proximal point algorithms in~\cite{WDS2013}.
We established a second order setup for
denoising $\mathbb S^1$-valued data based on cyclic proximal point algorithms
in~\cite{bergmann2014second}. Inpainting for $\mathbb S^1$-valued data was considered in 
the authors' conference proceeding~\cite{BW15}.

Data consisting of combined cyclic and linear space components appear in various contexts.
For example, such data appear when dealing with nonlinear color spaces such as HSV, HSL, HSI or HCL.
Such data also appear in the context of optical flows. 
When considering the flow vectors between consecutive images in polar coordinates which means separating magnitude and direction, the resulting data takes its values in \(\mathbb R\times \mathbb S^1\). In this context, this approach is natural and interesting and seems promising to improve the results obtained with the usual $\mathbb R^2$-valued approach. In principle, whenever data is given as vectors of polar
coordinates, we are in the combined cyclic and real-valued setup of data in $(\mathbb S^1)^m \times \mathbb R^m$ considered in this paper.

\paragraph{Contributions.}\par
In this paper, we consider inpainting, denoising as well as combined inpainting and denoising problems for combined cyclic and linear space valued images in $(\mathbb S^1)^m \times \mathbb R^n$ based on a second order TV-type formulation.
We consider two variational models:
the first model deals with the noise free situation whereas the second one also considers the noisy case combining denoising and inpainting (including pure denoising by specifying the inpainting area as the empty set). 
In our nonlinear setting, these higher order approaches avoid unwanted staircasing effects as well.
For combined cyclic and linear space data, 
we derive solvers for these variational problems based on cyclic proximal point algorithms.
We provide a convergence analysis for both the noisy and the noise free model
based algorithms developed in this paper. 
For both algorithms, we show the convergence to a minimizer under certain restrictions which are 
typical when dealing with nonlinear data.
	In particular, we assume the data to be dense enough meaning that they are locally (and not necessarily globally) nearby.
We apply our algorithms to denoising, inpainting and combined denoising and inpainting 
in the nonlinear HSV color space.
Furthermore, we apply our algorithms for denoising frames in volumetric phase-valued data -- in our case, frames of a 2D film. Our approach is based on utilizing the neighboring $k$ frames to incorporate the temporal neighborhood. 
The idea generalizes to arbitrary data spaces and volumes consisting of layers of 2D data.

The novelties of the present work in relation to the authors previous work \cite{bergmann2014second,BW15}
are as follows. 
(i) In contrast to \cite{bergmann2014second,BW15} we consider the more general, practically relevant, data spaces $(\mathbb S^1)^m \times \mathbb R^n$ here. In contrast to general manifolds, these product spaces still bear enough structure relevant for our purposes. We point out that both, the algorithmic and analysis part, 
are more involved than only component-wise considering the $\mathbb S^1$-valued or real-valued situation.
(ii) Concerning the algorithmic part, we compute proximal mappings of constrained problems arising
in inpainting situations in this work. This is even new for $\mathbb S^1$ data. In~\cite{BW15}, we used a less natural  projection approach 
generalizing~\cite{bergmann2014second}.
(iii) Concerning the analysis, we here include an inpainting setup. The conference proceeding~\cite{BW15} does not contain an analysis
and~\cite{bergmann2014second} considers functionals for denoising $\mathbb S^1$ valued data.
A more detailed discussion may be found at the end of the paper.

\paragraph{Outline of the paper.}\par 
In Section~\ref{sec:Models} we introduce the variational models 
we consider for inpainting and denoising of combined cyclic and linear space
data in this paper.
We start with vector space data in
Subsection~\ref{subsec:ModelVectorSpace}; then we define 
absolute differences for combined cyclic and vector space data
in Subsection~\ref{subsec:DefAbsDifferences} which allow us to derive
the corresponding variational models for combined cyclic and vector space data.
This is done in Subsection~\ref{subsec:InpaintingCombined}
for both inpainting noise free combined cyclic and vector space data
as well as inpainting and denoising combined cyclic and vector space data.
In Section~\ref{sec:Algorithms} we develop algorithms for minimizing the variational models introduced previously.
These algorithms base on the cyclic proximal point algorithm we present in Subsection~\ref{subsec:cppaGeneralExplanation}.
We present explicit formulas for the proximal mappings needed for inpainting
in Subsection~\ref{subsec:proxiesNeeded4Inpainting}. Using these explicit representations,
we derive a cyclic proximal point algorithm for inpainting both the noisy and noise free combined cyclic and vector space data in
Subsection~\ref{subsec:InpaintDenoiseCPPA}.
The convergence analysis of both algorithms is the topic of Section~\ref{subsec:Convergence}.
Finally, in Section~\ref{sec:Applications} we apply the derived algorithms to various concrete situations. We consider denoising data living in the nonlinear HSV color space in Subsection~\ref{subsec:NLcolorspaceAppl}.
Then we consider inpainting for noise-free as well as noisy data in such color spaces in Subsection~\ref{subsec:Inpainting}.
Finally, we apply our algorithms for denoising frames of a $\mathbb{S}^1$-valued 2D film in Subsection~\ref{sec:denoisingVolumes}.
\section{Second order variational models for inpainting and denoising combined cyclic and linear space data}
\label{sec:Models}
%
%
%
In this section we derive models for denoising, inpainting as well as simultaneous inpainting and denoising data having cyclic and linear space components.
In Subsection~\ref{subsec:ModelVectorSpace}, we first concentrate on introducing the considered models based on first and second order absolute finite differences restricting
to the linear space setting. In
Subsection~\ref{subsec:DefAbsDifferences} we obtain suitable definitions
for absolute differences for combined cyclic and vector space data.
In Subsection~\ref{subsec:InpaintingCombined}, we use these definitions
to obtain inpainting and simultaneous inpainting and denoising models for
combined cyclic and linear data. In particular, denoising is covered by
considering the empty set as inpainting region.

Let us fix some notations. We denote by
\((f_{i,j})_{(i,j)\in\Omega}
\),
\(\Omega_0 \coloneqq \{1,\ldots,N\}\times\{1,\ldots,M\}\) images of size
\(N\) by \(M\), which can also be seen as functions \(f(i,j)\) on the image
domain \(\Omega_0\). For any subset \(\Omega\subset\Omega_0\) of pixel indices
\(\Omega^C\coloneqq \Omega_0\backslash\Omega\) denotes the complement of
\(\Omega\) in \(\Omega_0\).
Each entry \(f_{i,j}\in\mathcal X\) is called a pixel
value, where \(\mathcal X\) is some data space.
Our main focus is the \((m+n)\)-dimensional space
\(\mathcal X \coloneqq (\mathbb S^1)^m\times\mathbb R^n\). On a data space
\(d_{\mathcal X}(x,y),D_1(x,y),D_2(x,y,z)\), and \(D_{1,1}(w,x,y,z)\) denote the
distance on~\(\mathcal X\), the absolute first, absolute second, and absolute
second order mixed differences. If the space is clear from the context or the
setting holds for all spaces, we omit the \(\mathcal X\).
The elements \(x = (x_1,\ldots,x_n)^\tT\in\mathbb R^n\) are column vectors, where \(\cdot^\tT\) denotes
the transposition. For two vectors \(x,y\in\mathbb R^n\)
we denote by \(\langle x,y\rangle \coloneqq x^\tT y\) the inner product. By
\((x)_{2\pi} \in [-\pi,\pi)\) we denote the elementwise modulo operation,
i.e.\ the unique value such that \(x = 2\pi k + (x)_{2\pi}\), \(k\in\mathbb Z^n\).
Finally, for matrices \(x = (x_{i,j})_{i=1,j=1}^{n,d}\in\mathbb R^{n\times d}\) we
denote the columns by \(x^{(j)}\).

\subsection{Inpainting and denoising vector space data}
\label{subsec:ModelVectorSpace}

The Rudin-Osher-Fatemi (ROF) functional~\cite{ROF92}
\[
\sum_{i,j} ( f_{i,j} - x_{i,j} )^2
	+ \alpha \sum_{i,j}  \norm{\nabla x_{i,j}}%
	, \quad \alpha > 0,
\]
is one of the most well known and most popular functionals in variational
image processing. In its penalized form, it consists of two terms: the first
term  measures the distance to the data $f$
the second term is a TV regularizer
where $\nabla$ denotes the discrete gradient operator, usually implemented as
first order forward differences in vertical and horizontal directions.
Both an isotropic version using the euclidean length or \(2\)-norm and an anisotropic version employing the sum of absolute values or \(1\)-norm for the second term are widely used. In this work we will restrict ourselves to the anisotropic version.
In this form, the ROF Model is typically used for denoising purposes.
To avoid the appearing staircasing effect, often higher order and, in
particular, second order differences (respectively, derivatives, in a continuous domain setting) are employed \cite{CL97,BKP09,CEP07,CMM00,DSS09,DWB09,HS06,LBU2012,LLT03,LT06,Sche98,SS08,SST11}.

\paragraph{Denoising.}\par
For pure denoising we consider the discrete second order TV-type functional
\begin{equation}\label{eq:2DTVfunctional}
\begin{split}
	J(x) &=
F(x; f)
+ \alpha \operatorname{TV}_1(x)
\\&\qquad
+ \beta \operatorname{TV}_2(x)
+ \gamma \operatorname{TV}_{1,1}(x),
\end{split}
\end{equation}
where \(x,f\) are images defined on the image domain\(\Omega_0\) denotes the image domain. The data values \(f_{i,j}\) itself live in a certain data space. Then the data term $F(x;f)$ for given data $f$ reads
\begin{align}
F(x;f)
&=
\frac{1}{2}
\sum_{i,j=1}^{N,M}
d(f_{i,j},x_{i,j})^2,\label{eq:2DFdist}
\end{align}
where $d$ is a distance on the data space.
For data living in a vector space, $d(f_{i,j},x_{i,j}) = \norm{f_{i,j}-x_{i,j}}$ is an appropriate choice.
In the pure linear space data situation, the above quadratic data term \eqref{eq:2DFdist} 
	corresponds to a Gaussian noise model. For other types of noise, different data terms are more appropriate; e.g., for Laplacian noise the term $F(x;f)=\sum_{i,j=1}^{N,M}d(f_{i,j},x_{i,j})$ is appropriate; cf. also \cite{PS14}. 

The first order difference component $\alpha\operatorname{TV}_1(x)$ is given by
\begin{equation}\label{eq:DefFirstDiff4VectorData}
	\begin{split}
		\alpha\operatorname{TV}_1(x)
		&=
		\alpha_1\sum_{i,j=1}^{N-1,M} D_1(x_{i,j},x_{i+1,j})
		\\
		&\qquad
		+
		\alpha_2\sum_{i,j=1}^{N,M-1} D_1(x_{i,j},x_{i,j+1})
		\\
		&\qquad
		+
		\frac{\alpha_3}{\sqrt{2}}\sum_{i,j=1}^{N-1,M-1} D_1(x_{i,j},x_{i+1,j+1})
		\\
		&\qquad
		+
		\frac{\alpha_4}{\sqrt{2}}\sum_{i,j=1}^{N-1,M-1} D_1(x_{i,j+1},x_{i+1,j}).
	\end{split}
\end{equation}
Again, for a vector space any norm of the ordinary absolute first order difference
$D_1(x_{i,j},x_{i+1,j}) = \norm{x_{i,j}-x_{i+1,j}}$ is an appropriate choice. The first order TV term incorporates horizontal, vertical and both diagonal differences.
The diagonals are incorporated to reduce unwanted anisotropy effects and are scaled by \(\frac{1}{\sqrt{2}}\) to take the length of the diagonal on the pixel grid, i.e. the distance of two pixels, into account.
We note that
$J'(x) = F(x; f) + (\alpha_1,\alpha_2,0,0)\operatorname{TV}_1(x)$
is just the vector version of the anisotropic discrete
ROF functional above.
Using the notation
\begin{align}\label{eq:DefSecondDiff4VectorData}
D_2(x,y,z) &= \lVert x-2y+z\rVert\\
\intertext{and}
D_{1,1}(x,y,u,v) &= \lVert x-y-u+v\rVert,\nonumber
\end{align}
for a norm of the standard second order differences for vector space data,
the second order difference component,
consisting of a horizontal and vertical component $\beta\operatorname{TV}_2(x)$
as well as a diagonal component $\gamma \operatorname{TV}_{1,1}(x),$ is given by
\begin{align}
	\phantom{\gamma \operatorname{TV}_{1,1}(x)}
	&\begin{aligned}
	\mathllap{\beta\operatorname{TV}_2(x)}
		&=
		\beta_1\!\sum_{i=2,j=1}^{N-1,M}\!\!
				D_2(x_{i-1,j},x_{i,j},x_{i+1,j})
		\\&\qquad+
			\beta_2\!\sum_{i=1,j=2}^{N,M-1}\!\!
				D_2(x_{i,j-1},x_{i,j},x_{i,j+1}),
		\label{eq:TV2iso}
	\end{aligned}\\
	&\begin{aligned}
		\mathllap{\gamma\operatorname{TV}_{1,1}(x)}
		&=
		\gamma\!\!\!\!\sum_{i,j=1}^{N-1,M-1}\!\!\!\!
			D_{1,1}(x_{i,j},x_{i+1,j},x_{i,j+1},x_{i+1,j+1}).
			\label{eq:TV2mix}
	\end{aligned}
\end{align}
We note that similar to the diagonal differences in the $\operatorname{TV}_1$
part the diagonal component $\gamma \operatorname{TV}_{1,1}(x)$, which is based
on the  mixed second order difference $D_{1,1}$, reduces unwanted anisotropy
effects for the second order part.
Actually, $D_{1,1}(x_{i,j},x_{i+1,j},x_{i,j+1},x_{i+1,j+1})$ may be interpreted
as a diagonal difference:
averaging $x_{i,j+1}$ and $x_{i+1,j}$ yields an estimate for the non-grid value
$m=``x_{i+1/2,j+1/2}$''. 
Then $D_2(x_{i,j}, m, x_{i+1,j+1})
=
D_{1,1}(x_{i,j},\\
 x_{i+1,j}, x_{i,j+1}, x_{i+1,j+1})$
plays the role of a second order diagonal difference. Interchanging the roles
of $x_{i,j+1},\\x_{i+1,j}$ and  $x_{i,j},x_{i+1,j+1}$ yields
the other diagonal. Hence, diagonal differences are already incorporated in the
second order term and we do not have to consider terms of the
form $D_2(x_{i-1,j-1}, x_{i+1,j+1}, x_{i,j})$
for the reduction of an\-iso\-tropy effects.
The model parameters $\alpha_1,\alpha_2,\alpha_3,\alpha_4,\\\beta_1,\beta_2,\gamma$ regulate the influence of the different TV terms.

One main reason for considering this anisotropic model is that it is computationally feasible 
 for the nonlinear space of combined cyclic and vector space data considered in this paper
 as we will see later on. To our knowledge there are no previous algorithms dealing with any kind of second order TV-like problems in this nonlinear situation ---neither in the isotropic nor in the anisotropic setup. As explained, we employ diagonal terms 
 to milden unwanted effects caused by the anisotropic formulation.

Next, we consider suitable modifications of the above functional to obtain models for the inpainting problem with noisy and noiseless data.
We start by first formulating the inpainting problem.

\paragraph{Inpainting problem in the presence of noise.}\par
Given an image domain \(\Omega_0 = \{1,\ldots,N\} \times \{1,\ldots,M\},\) an inpainting
region $\Omega\subset\Omega_0$ is a subset of the image domain \(\Omega_0\),
where the pixel values \(f_{i,j}\), \((i,j) \in \Omega\), are lost.
The noiseless or noisy inpainting problems now consist of finding
a function $x$ defined on \(\Omega_0\) from data $f$ given on the complement $\Omega^C$ of the inpainting region, such that $x$ is a suitable extension to~$f$ onto~$\Omega$ and for the second case additionally denoised.

\paragraph{Inpainting without presence of noise.}\par
To deal with the noiseless situation, we consider the following modification of the functional $J$ given by~\eqref{eq:2DTVfunctional}.
Since the data is assumed to be noiseless, we add the constraint that the target variable agrees with the data on the complement of the inpainting region. Furthermore, the data term considers only those indices for which actually data are available. This eliminates the data term from the functional.
More precisely, the second order variational inpainting problem considered in this paper reads for a vector space as
\begin{equation}\label{eq:2DTVfunctionalInpainting}
\begin{split}
	\argmin_{x
	}
	&\
	\alpha \operatorname{TV}_{1}(x)
	+ \beta \operatorname{TV}_{2}(x)
	+ \gamma \operatorname{TV}_{1,1}(x)%
	,
	\\&
	\text{ subject to }
	x_{i,j}=f_{i,j}\text{ for all }(i,j)\in \Omega^C
.
\end{split}
\end{equation}
The TV terms $\operatorname{TV}_{1},$ $\operatorname{TV}_{2},$
$\operatorname{TV}_{1,1}$ are defined by~\eqref{eq:DefFirstDiff4VectorData}, \eqref{eq:TV2iso} and~\eqref{eq:TV2mix}
using the difference terms \(D_1,D_2,D_{1,1}\) based on a norm in the vector space.
Due to the constraint they actually only act on those difference terms
that affect an entry in the inpainting region.

\paragraph{Second order TV formulation of the inpainting problem for noisy data.}\par
For the inpainting problem in presence of noise the requirement of equality on $\Omega^C$ is replaced by $x$ being a suitable, i.e., smooth approximation. In this case, we search for a minimizer of
the following second order TV functional for cyclic data for inpainting:
\begin{equation}\label{eq:2DTVfunctionalInpaintingWithNoise}
\begin{split}
J_{\Omega}(x) &=
F_{\Omega^C}(x; f)
+ \alpha \operatorname{TV}_{1}(x)
\\&\qquad+ \beta \operatorname{TV}_{2}(x)
+ \gamma \operatorname{TV}_{1,1}(x),
\end{split}
\end{equation}
where for any subset \(B\subset\Omega_0\) of the image domain, we define
\begin{equation}\label{eq:DataTerm4Inpainting}
F_{B}(x; f) \coloneqq \frac{1}{2}\sum\nolimits_{(i,j)\in B} d(x_{i,j},f_{i,j})^2.
\end{equation}
This means we use a data term that enforces similarity to the
given data \(f\) on \(\Omega^C\) while applying a regularization based
on~\eqref{eq:DefFirstDiff4VectorData},\eqref{eq:TV2iso}, and~\eqref{eq:TV2mix} for the whole image domain.
Specifying the inpainting area as the empty set, we obtain the pure denoising problem~\eqref{eq:2DTVfunctional}.

\subsection{Absolute differences for combined cyclic and vector space data}
\label{subsec:DefAbsDifferences}

In order to implement the above variational inpainting
problem~\eqref{eq:2DTVfunctionalInpainting} and the simultaneous inpainting and
denoising problem~\eqref{eq:2DTVfunctionalInpaintingWithNoise} for combined
cyclic and vector space data, we have to find suitable difference
operators $D_1, D_2, D_{1,1}$ for data consisting of combined cyclic and linear space components.
In order to do so, we first find suitable definitions for vectors of cyclic data.
Then, we combine these definitions with those for the linear space case in a way suitable to
the space of interest in this paper.
We use the symbols
\(D_\bullet\)
already introduced in~\eqref{eq:DefFirstDiff4VectorData}
for the linear space case, also for the case of vectors of cyclic and combined data. We further unify the notation to \(D(\cdot\,;\,w)\) for different weights \(w\).
This overload is employed to avoid additional notation and should not cause confusion since the space under consideration will be clear from the context.

Let \(w = (w_j)_{j=1}^d\in\mathbb R^d\), \(w \neq 0\), be a vector with
\(\sum_{j=1}^d w_j = 0,\) and call such a vector \(w\) a \emph{weight}.
Special cases are the \emph{binomial coefficients with alternating signs}
\[
\textstyle b_d = \left( (-1)^{j+d-1} { d\choose {j-1}} \right)_{j=1}^{d+1}.
\]
For the vectors \(x^{(j)}\in\mathbb R^n\), \(j=1,\ldots,d\), we employ the
matrix~\(x \coloneqq \bigl(x^{(1)},\ldots,x^{(d)}\bigl)\in\mathbb R^{n\times d}\) to define the \emph{absolute finite difference} \(D\) for these 
vectors
with respect to the weight~\(w\in\mathbb R^d\) by
\[
D(x;w)
= \norm{xw}
= \norm[\Big]{\sum_{j=1}^{d} w_jx^{(j)}}.
\]
For \(w=b_d\), we obtain the forward differences of order $d$ for the
vectors \(x^{(1)},\ldots,x^{(d+1)}\), i.e.\ \[
D_d(x)
\coloneqq D(x;b_d)
= \norm[\Bigl]{\sum_{j=1}^{d+1} (-1)^{j+d-1} { \textstyle{ d\choose {j-1}}} x^{(j)}}.
\]
Another useful weight used in this paper is \(w=b_{1,1} \coloneqq (-1,1,1,-1)\).
We denote the corresponding finite absolute difference by \(D_{1,1}(x) \coloneqq D(x;b_{1,1})\).
\begin{example}\label{ex:diff2-2d}
	For three points \(x_1,x_2,x_3\in\mathbb R^n\) and \(w=b_2 = (1,-2,1)^\tT\)
	the second order absolute difference is given
	by~\(D_2(x_1,x_2,x_3) = \norm{x_1-2x_2+x_3}\), cf.~\eqref{eq:DefSecondDiff4VectorData}.
	This can be interpreted as measuring the distance from~\(y\) to the
	midpoint~\(m_x \coloneqq \frac{1}{2}(x_1+x_3)\)
	of the line segment connecting~\(x_1\) and~\(x_3\); more precisely, we have
	\(D_2(x_1,x_2,x_3) = 2\norm{\tfrac{1}{2}(x_1+x_3)-x_2}\). The situation is shown in Figure~\ref{fig:2Ddiff2-base} for~\(n=2\), but also
	illustrates the situation for general \(n>2\) in the plane defined
	by~\(x_1,x_2,x_3\). For \(n=1\) the situation simplifies to \(x_2\) always lying on
	the line ---though not necessarily the segment--- connecting \(x_1\) and \(x_3\).
\end{example}
\begin{figure}[tbp]%
	\centering
	\includegraphics{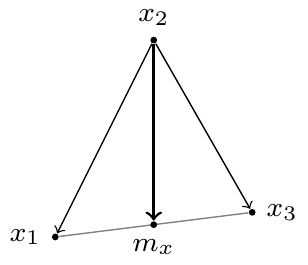}
	\caption[]{The three points \(x_1,x_2,x_3\in\mathbb R^2\) illustrate the
		multivariate finite difference \(D_2(x_1,x_2,x_3) = 2\norm{m_x-x_2}\), i.e. the distance of \(x_2\) to the
		midpoint~\(m_x \coloneqq \tfrac{1}{2}(x_1+x_3)
		\) of \(x_1\) and \(x_3\). This measures how “near” they are to lying equally distributed on a line segment in the right order.}
	\label{fig:2Ddiff2-base}
\end{figure}

We consider the cyclic case next. Let~\(\mathbb S^1\) denote the unit circle
\(\mathbb S^1 \coloneqq \{p_1^2 + p_2^2 = 1: p = (p_1,p_2)^\tT \in \mathbb R^2\}\)
endowed with the \emph{geodesic} or \emph{arc length distance}
\[
d_{\mathbb S^1} (p,q) = \arccos\,\langle p,q \rangle, \quad p,q\in\mathbb S^1.
\]
Given a base point $q \in \mathbb S^1$, the
\emph{exponential map}~$\exp_q\colon\\\mathbb R \rightarrow \mathbb S^1$ from the tangent
space~$T_q\mathbb S^1 \simeq \R$ of $\mathbb S^1$ at $q$ onto $\mathbb S^1$ is
defined by
\[
\exp_q(x) = R_x q, \qquad R_x := \begin{pmatrix}
	\cos x & -\sin x\\
	\sin x & \cos x
\end{pmatrix}.
\]
This map is $2\pi$-periodic, i.e.,
$\exp_q(x) = \exp_q((x)_{2\pi})$ for any $x \in \R$,
where $(x)_{2\pi}$ is the unique point such that
\begin{equation}\label{eq:DefRoundBracket}
x = 2\pi k + (x)_{2\pi} \quad \text{with} \quad (x)_{2\pi}  \in [-\pi,\pi), k\in \mathbb Z.
\end{equation}
If we fix \(q\), we obtain a
\emph{representation system} of \(\mathbb S^1\), i.e., \(\exp_q\) is a
bijective map where \(\exp_q(0) = q\) and there is a unique~\(x\in[-\pi,\pi)\)
for each \(p\in\mathbb S^1\) such that \(\exp_q(x) = p\).
A vector \(x\in [-\pi,\pi)^m\) represents a point in \((\mathbb S^1)^m\)
by component-wise application, and for a point \(q\in(\mathbb S^1)^m\), the map
\(
\exp_q\colon[-\pi,\pi)^m\rightarrow (\mathbb S^1)^m
\),
\(
\exp_{q}(x) \coloneqq (\exp_{q_1}(x_1),\ldots,\exp_{q_m}(x_m))^\tT,
\)
is bijective where the
properties from above hold component-wise.
A distance measure on \((\mathbb S^1)^m\) is given by
\[
d_{(\mathbb S^1)^m} (p,q)
= \norm[\big]{\bigl(\arccos\,\langle p_i, q_i\rangle \bigr)_{i=1}^m}.
\]
In the following, we introduce higher order differences
on~\((\mathbb S^1)^m\). We employ the representation system \(\mathbb S^1 \cong [-\pi,\pi)\) induced by using an arbitrary but fixed exponential map.
Let \(x^{(j)}\in [-\pi,\pi)^m\), \(j=1,\ldots,d\).
Using the notation \(x \coloneqq (x^{(1)}, \ldots, x^{(d)}) \in [-\pi,\pi)^{m\times d}\),
the \emph{absolute cyclic difference} of~\(x^{(1)},\ldots,x^{(d)}\) with respect
to a weight \(w\in\mathbb R^d\backslash\{0\}\) is defined as
\begin{equation*}
D(x;w)
\coloneqq
\min_{\alpha\in\mathbb R^m}
D([(x^{(1)}+\alpha,\ldots,x^{(d)}+\alpha)]_{2\pi};w),
\end{equation*}
where \([y]_{2\pi}\) for some \(y\in(\mathbb S^1)^m\) is multivalued and its $i$th component $([y]_{2\pi})_i$ is given by
\begin{equation} \label{eq:DefMultivaluedBracket}
([y]_{2\pi})_i =
   \begin{cases}
   (y_i)_{2\pi},    &  \text{ if } y_i\neq (2z+1)\pi \text{ for all } z\in\mathbb Z,     \\
    \pm\pi          & \text{ else.}
   \end{cases}
\end{equation}
This definition may seem a bit technical at first glance. However, it allows for two
points~\(x^{(i)},$ $x^{(j)},$ $i,j\in\mathbb I_d \coloneqq \{1,\ldots,d\}\), having the
same value \(x_l^{(i)}=x_l^{(j)}\) in one component~\(l\in\{1,\ldots,m\}\) to be treated differently, cf. the definition for \(\mathbb S^1\) in \cite[Section 2]{bergmann2014second}.
In fact, we may choose any \(q\in(\mathbb S^1)^m\) as a base point for our
representation system, which shifts any set of points given with respect
to~\(\exp_{q'}\) by a fixed value of \(\alpha \coloneqq \exp_q^{-1}(q')\).
When the shift by~\(\alpha\in\mathbb R^m\) is small enough, such that no
component of~\(x\) is affected by the component-wise application
of~\([\cdot]_{2\pi}\), both representation systems yield the same value. Using this notation we can simplify both the definition of the second order difference and the proximal mappings derived later on.
The minimum in the definition of the difference simplifies to
\begin{equation}\label{eq:xkshifts}
D(x;w)
= \min_{k\in\mathbb I_d^m}
D([(x^{(1)}-x_k+\pi,\ldots,x^{(d)}-x_k+\pi)]_{2\pi};w)
,
\end{equation}
where \(x_k \coloneqq \bigl(x^{(j)}_{k_j}\bigr)_{j=1}^m\). This is illustrated in
Figure~\ref{fig:shiftplane} for three points \(x,y,z\in(\mathbb S^1)^2\).
	We note that this definition contains the notion introduced in \cite{bergmann2014second} for $\mathbb S^1$ as a special case.
\begin{figure*}[tbp]\centering
	\begin{subfigure}[t]{.49\textwidth}\centering
		\includegraphics{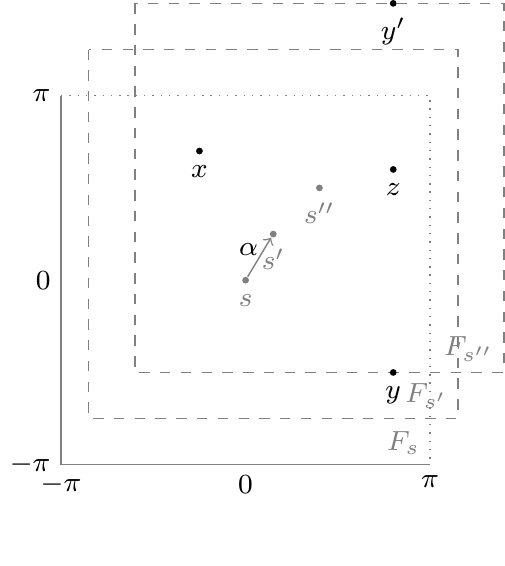}
		\caption{Different shifts, where the first two, \(F_s, F_{s'}\) yield the
			same constellation of \(x,y,z\).}\label{subfig:alphacases}
	\end{subfigure}
	\begin{subfigure}[t]{.49\textwidth}\centering
		\includegraphics{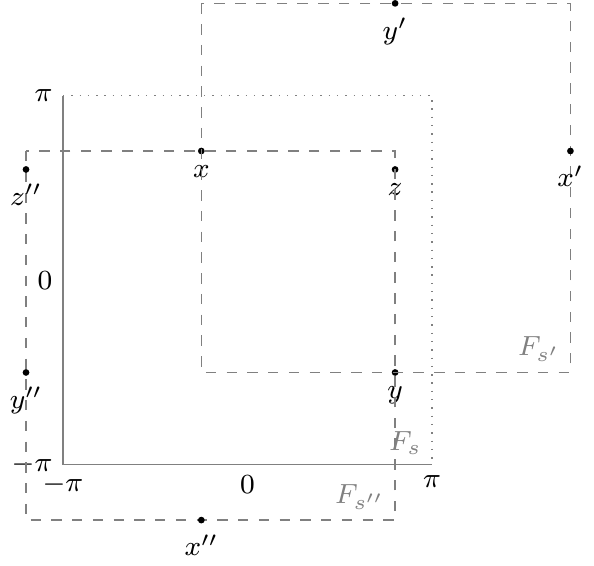}
		\caption{Three representation systems with points \(x,y,z\) at \(\pm\pi\)
			in at least one dimension.}\label{subfig:bordercases}
	\end{subfigure}
	\caption{For given points \(x,y,z\in[-\pi,\pi)^2\) in a representation
		system \(F_s\), i.e., with base point \(s\) the two subfigures illustrate different shifts:~(\subref{subfig:alphacases}) shifts by
		arbitrary \(\alpha\in\mathbb R^2\), e.g.,
		to \(s' = s+\alpha\) yielding the same value of \(D_2\) , and~(\subref{subfig:bordercases})
		cases yielding different values for the second order difference and a minimum occurring for \(x,y,z\) at the borders of the representation system.
	}
	\label{fig:shiftplane}
\end{figure*}

Finally, we come to the space of interest in this paper
which is \(\mathcal X \coloneqq (\mathbb S^1)^m\times \mathbb R^n\).
In this space a vector~\(x=(x_i)_{i=1}^{n+m}\) consists of two
parts: the phase-valued~\(x_\mathbb S \coloneqq (p_i)_{i=1}^m\in(\mathbb S^1)^m\)
and the real valued
components~\(x_{\mathbb R} = (x_{i})_{i=m+1}^{m+n}\in\mathbb R^n\)
of \(x\in\mathcal X\). In the following, we will use any representation system in order to write the cyclic components~\(x_{\mathbb S}\) as a vector \(x_{\mathbb S}\in[-\pi,\pi)^m\). This also allows for \(x\in\mathcal X\) to be seen as a vector, where the first \(m\) components are restricted to \([-\pi,\pi)\) and the remaining \(n\) ones are real-valued.
The distance of two points~\(x,y\in\mathcal X\) on this product space is given
by
\begin{equation}
d_{\mathcal X}(x,y)
= \sqrt{%
	\norm{x_{\mathbb R}-y_{\mathbb R}}^2
	+ d_{(\mathbb S^1)^m}(x_{\mathbb S},y_{\mathbb S})^2}.
\end{equation}
For a set of points \(x^{(1)},\ldots,x^{(d)}\in\mathcal X\),
using the notation \(x=(x^{(1)},\ldots,x^{(d)})\) as before, the \emph{finite difference for cyclic and noncyclic data} with respect to a weight \(w\in\mathbb R^{d}\backslash\{0\}\) is defined by
\begin{equation}
D(x;w) \coloneqq \sqrt{%
	D(x_{\mathbb R};w)^2 + D(x_{\mathbb S};w)^2
}.
\end{equation}
We further introduce the short hand notations
\begin{equation}\label{eq:diffCycDiff}
D_d(x) = D(x,b_d),\quad x\in\mathcal X^{d+1},\ d\in\mathbb N,
\end{equation}
to denote the corresponding absolute finite differences of order \(d\).
Furthermore we introduce ---with a slight abuse of the difference notation--- the \emph{second order mixed difference} of four points, e.g. given on a \(2\times 2\) subset of the pixel grid \(\mathcal \Omega_0\)) by~\(D_{1,1}(x) \coloneqq D(x;b_{1,1})\) for \(x\in\mathcal X^{4}\),
with \(b_{1,1} = (-1,1,1,-1)^\tT\).
For the weights corresponding to first and second order differences,
we have a particularly nice representation, which is given by the following Lemma.
\begin{lemma}\label{lem:RepresentationCombinedDifferences}
For~\(w\in\{b_1,b_2,b_{1,1}\}\) and~$x \in \mathcal X^d $  where~$d$ denotes the length of~\(w\), we have
\begin{equation}\label{eq:Dxwmodulo}
D(x,w)^2 = \norm{(x_{\mathbb S}w)_{2\pi}}^2 + \norm{x_{\mathbb R}w}^2 ,
\end{equation}
\end{lemma}
\begin{proof}
For the real-valued components there is nothing to show. Hence we may restrict to \(n=0\) (which corresponds to purely cyclic data)
and thus have to show that $D(x,w) = \norm{(xw)_{2\pi}}$ for \(x\in(\mathbb S^1)^{n\times d}\).
To this end we apply Proposition 2.5 of~\cite{bergmann2014second} to each row \(x\) and conclude
the validity of~\eqref{eq:Dxwmodulo}.
\end{proof}
\subsection{Inpainting combined cyclic and vector space data}
\label{subsec:InpaintingCombined}
We can now apply the definition of absolute differences for combined cyclic and
vector space data we derived in Section~\ref{subsec:DefAbsDifferences} to obtain variational models for inpainting and simultaneous inpainting and denoising. This extends the models for the Euclidean differences in Subsection~\ref{subsec:ModelVectorSpace} using a first order TV term~\eqref{eq:DefFirstDiff4VectorData} as well as to the second order TV terms~\eqref{eq:TV2iso}
and~\eqref{eq:TV2mix} to a more general setting. The noiseless inpainting model now reads
\begin{equation}\label{eq:Functional4InpaintingNoiselessCombCyclicVec}
\begin{split}
	\argmin_{{x}\in\mathcal X^{N\times M}}&\
\alpha \operatorname{TV}_{1}(x)
+ \beta \operatorname{TV}_{2}(x)
+ \gamma \operatorname{TV}_{1,1}(x),\\
&\text{ subject to }
x_{i,j}=f_{i,j}\text{ for all }(i,j)\in \Omega^C
.\end{split}
\end{equation}
Here, $\operatorname{TV}_{1}$ is defined by~\eqref{eq:DefFirstDiff4VectorData} incorporating the first order absolute differences $D_1$ given by~\eqref{eq:diffCycDiff}
and
the second order TV terms
$\operatorname{TV}_{2}$,
$\operatorname{TV}_{1,1}$ are defined by
\eqref{eq:TV2iso} and~\eqref{eq:TV2mix} employing
the second order absolute differences $D_2,$ $D_{1,1}$ given
by~\eqref{eq:diffCycDiff}, respectively.

Proceeding similarly, we get a variational formulation of the inpainting problem for noisy combined cyclic and vector space data by computing a minimizer of
\begin{equation}
	\label{eq:Functional4InpaintingNoisyCombCyclicVec}
	\begin{split}
	J_{\Omega}(x) &=
	F_{\Omega^C}(x; f)
	+ \alpha \operatorname{TV}_{1}(x)
	\\&\qquad+ \beta \operatorname{TV}_{2}(x)
	+ \gamma \operatorname{TV}_{1,1}(x).
	\end{split}
\end{equation}
Here the data term is given by~\eqref{eq:DataTerm4Inpainting}
using the distance function on $\mathcal X = (\mathbb S^1)^m\times \mathbb R^n.$
As in the noiseless situation,
$\operatorname{TV}_{1}$ is defined by~\eqref{eq:DefFirstDiff4VectorData} again incorporating the first order absolute cyclic differences $D_1$
and the second order TV terms
$\operatorname{TV}_{2}$,
$\operatorname{TV}_{1,1}$ are defined by
\eqref{eq:TV2iso} and~\eqref{eq:TV2mix} employing
the second order absolute cyclic differences $D_2,$ $D_{1,1}$ from \eqref{eq:diffCycDiff}, respectively.
\section{Algorithms for inpainting and denoising combined cyclic and linear space data}
\label{sec:Algorithms}
%
%
In this section, we derive algorithms to solve the inpainting
problem~\eqref{eq:Functional4InpaintingNoiselessCombCyclicVec}, and the combined inpainting and
denoising problem~\eqref{eq:Functional4InpaintingNoisyCombCyclicVec}. Note that the latter includes the denoising of
combined cyclic and vector space data for the case of \(\Omega=\emptyset\), an
empty inpainting set. These algorithms are based on a cyclic proximal point
algorithm  whose concept we recall in
Section~\ref{subsec:cppaGeneralExplanation}.
We derive explicit formulas for the proximal mappings that are needed for inpainting
and denoising of such combined data in
Section~\ref{subsec:proxiesNeeded4Inpainting}.
Using these explicit representations, we derive a cyclic proximal point
algorithm for inpainting noiseless combined cyclic and vector space data
and similarly for simultaneously inpainting and denoising data
in Section~\ref{subsec:InpaintDenoiseCPPA}. This also includes an efficient
choice for the cycles involved. Finally in Section~\ref{subsec:Convergence}, we
prove convergence of our algorithm to a minimizer under certain conditions
that reflect the space-inherent non-convexity of the involved functionals.
%
%
%
\subsection{The cyclic proximal point algorithm}
\label{subsec:cppaGeneralExplanation}
For a closed, convex and proper functional \(\varphi\colon\mathbb R^n\to \mathbb R\cup\{\infty\}\) the \emph{proximal mapping} is given by
\begin{equation}\label{eq:DefProx}
			\prox_{\lambda \varphi}(f)
			\coloneqq
			\argmin_{x\in\mathbb R^n}
			\frac{1}{2} \norm[\big]{f-x}^2 + \lambda \varphi (x),
			\quad
			f\in\mathbb R^n,
\end{equation}
where \(\lambda > 0\) is a tradeoff or regularization parameter. The fixed points of \(\prox_{\lambda \varphi}(f)\) are minimizers of \(\varphi\). Hence, if the proximal mapping \(\prox_{\lambda \varphi}(f)\) can be computed in closed form, an algorithm for finding a minimizer is given by iterating
\[
	x^{(k)} = \prox_{\lambda \varphi}(x^{(k-1)}),\quad k=1,2,\ldots
\]
for some starting value \(x^{(0)}
\). This algorithm is called \emph{proximal point algorithm} (PPA) and was introduced by Rockafellar~\cite{Roc76}. It was recently extended to Riemannian manifolds~\cite{FO02} and also to Hadamard spaces~\cite{Bac13}.
Denoting the distance on a Riemannian manifold \(\mathcal M\) by \(d_{\mathcal M}\) the \emph{proximal mapping on a manifold} reads for a function \(\varphi\colon\mathcal M^n\to\mathbb R\) as
\begin{equation}\label{eq:DefProxM}
			\prox_{\lambda \varphi}(f)
			\coloneqq
			\argmin_{x\in\mathcal M^n}
			\frac{1}{2} d_{\mathcal M}(f,x)^2 + \lambda \varphi (x),
			\quad f\in\mathcal M^n.
\end{equation}
We note that on some manifolds, e.g.\ the spheres \(\mathbb S^d\) there is no definition of a (globally) convex function. Hence the minimizer might not be unique. This is for example the case when looking at our space \(\mathcal X^d\), as this is not even the case for cyclic data, cf.~\cite{bergmann2014second}.

If the function \(\varphi\) can be split into simpler parts, i.e.\ \( \varphi = \sum_{i=1}^c \varphi_i\), for which then individually the proximal mappings are known in closed form, a similar algorithm is given for a sequence \(\{\lambda_k\}_k\) of regularization parameters by
\begin{equation*}
	\begin{split}
	x^{(k + \frac{l}{c})} = \prox_{\lambda_k \varphi_{l}}
	\bigl(
	x^{(k + \frac{l-1}{c})}
	\bigr),&%
	\\l=1,\ldots,c,&\ k=1,2,\ldots,
	\end{split}
\end{equation*}
and it is called \emph{cyclic proximal point algorithm} (CPPA). Its formulation on Euclidean space is derived in~\cite{Ber11}, see also the survey~\cite{Ber10}.
It converges to a minimizer of \(\varphi\) if
\begin{equation}\label{eq:CPPAlambda}
		\sum\nolimits_{k=0}^\infty \lambda_k = \infty,
		\quad \text{and}
		\quad \sum\nolimits_{k=0}^\infty \lambda_k^2 < \infty.
\end{equation}
The concept of CPPAs for Hadamard spaces has been treated in in~\cite{Bac13a}.
A CPPA for TV minimization for manifolds and in Hadamard spaces has been derived in \cite{WDS2013}. For second order TV type problems, a CPPA to denoise $\mathbb{S}^1$ data
was derived in~\cite{bergmann2014second}.
A preliminary model, different to the one appearing in this paper, was applied to inpainting of $\mathbb{S}^1$ data in~\cite{BW15}. For manifold data in general, the main challenge  is to derive proximal mappings which are as explicit as possible.

\subsection{Proximal mappings for inpainting}
\label{subsec:proxiesNeeded4Inpainting}

Here we derive closed form expressions for the proximal mappings needed to make the cyclic proximal point algorithm from 
Section~\ref{subsec:cppaGeneralExplanation} work for the inpainting problems~\eqref{eq:Functional4InpaintingNoiselessCombCyclicVec}
and~\eqref{eq:Functional4InpaintingNoisyCombCyclicVec} in the nonlinear spaces considered in this paper.
In particular, we derive proximal mappings incorporating constraints directly. 

We first need some basic results on the linear case which involves vectors of real-valued data only.
To this end, we start with a generalization of~\cite[Lemma 3.1]{bergmann2014second}. We derive explicit expressions for the proximal mappings of functions living on linear spaces which are of the form
\begin{equation}
\varphi(x) = \norm{xw-a}, \quad a\in\mathbb R^n,
\end{equation}
where the target variable $x$ is a matrix in $\mathbb R^{n\times d}$, and
where \(d\) corresponds to the length of \(w\). 
The vector \(a\) introduces an offset. We employ the notation 
\(\norm{y}[\text{F}] = \sqrt{\sum_{i,j=1}^{n,d} \bigl(y_i^{(j)}\bigr)^2}\) to denote the
Frobenius norm of a matrix \(y\).

\begin{lemma}\label{lem:proxRn}
	Let~\(f = \bigl(f^{(1)},\ldots,f^{(d)}\bigr)\in\mathbb R^{n\times d}\) be a matrix whose columns \(f^{(i)}\) represent the data vectors,
    let \(0\neq w\in\mathbb R^d\) ($w$ not necessarily a weight),
	and \(\lambda > 0\) be given. For the functional
	\begin{align}\label{eq:def-E}
	E(x;f,a,w) = \frac{1}{2}\norm{f-x}[\text{F}]^2 + \lambda\norm{xw-a},
	\end{align}
	with target variable  $x \in \mathbb R^{n\times d}$, the minimizer \(\hat x\) is given by 
	\begin{align}\label{eq:minim-E}
	\hat x = f - msw^\tT,
	\end{align}
	where
	~\( s \coloneqq 
		\begin{cases}
		\frac{fw-a}{\norm{fw-a}}&\mbox{ if \( \norm{fw-a}\neq 0 \),}\\
		0 &\mbox{ else,}
	\end{cases}\)
	\\
	and
	~\(m \coloneqq \min\bigl\{\lambda\), \(
		\frac{\norm{fw-a}}{\norm{w}^2}\bigr\}\).
	The minimum $E(\hat x;f,a,w)$ is given by  
    \begin{align}\label{eq:E-Minimum}
		E(\hat x;& f,a,w)\\
		&= 
		\begin{cases}
		    \frac{1}{2}{\norm{fw-a}^2}
		         &\mbox{if \( m\leq \lambda \),}\\
		    \norm{w}^2\bigl(\frac{1}{2}\lambda^2 + \lambda    
		           (\norm[\big]{\frac{fw-a}{\norm{w}^2}} - \lambda))
		         &\mbox{otherwise.}
		\end{cases}
    \end{align}
	Furthermore, given data \(f,\tilde f\in\mathbb R^{n\times d}\),
	and different offsets \(a,\tilde a\in\mathbb R^n\), the
	following implication holds:
	\begin{align}\label{eq:CompareDifferentfanda}   
		\norm{fw&-a} < \norm{\tilde fw-\tilde a}\\
		&\implies \	
		\min_{x\in\mathbb R^{n\times d}} E(x;f,a,w) <  \min_{x\in\mathbb R^{n\times d}} E(x;\tilde f,\tilde a,w).
	\end{align}
\end{lemma}
\begin{proof}
	We first reduce the functional to be minimized to an equivalent problem
	without offset.
	By assumption there is an index $j$ such that $w_j \neq 0$,
    which allows us to write~\eqref{eq:def-E} as 
    \begin{align*}
    E(x;f,a,w) 
	= \frac{1}{2}
	\norm{f-x}[\text{F}]^2
	+ \lambda|w_j|\norm[\big]{\bigl(x-\frac{a}{w_j}e_j^\tT\bigr) \bigl(\frac{w}{w_j}\bigr)}.    
    \end{align*}
    Defining the new target matrix  $y \coloneqq x - \frac{a}{w_j} e_j^\tT$, 
    the new data matrix $g = f - \frac{a}{w_j} e_j^\tT,$
	the new regularizing parameter $\nu \coloneqq \lambda \abs{w_j}\rvert$,
	and the new (not necessarily weight) vector $v \coloneqq \frac{w}{w_j}$, we obtain the new problem
	\begin{align}\label{eq:ReducedProblemWithoutShift}
	F (y;g,v) = \frac{1}{2} \norm{g-y}[\text{F}]^2 + \nu \norm{yv},
	\end{align}	
	where the second term is free of an offset. The relation between minimizers \(\hat{x}\) of \(E\) and \(\hat{y}\) of \(F\) is given via
	$\hat{y} = \hat{x} - \frac{a}{w_j} e_j^\tT.$
	
	We now consider the problem~\eqref{eq:ReducedProblemWithoutShift} and first
	show~\eqref{eq:minim-E} for $F.$
	The corresponding statement for $E$ follows by carrying out the
	resubstitution.	
	If \(\norm{gv} = 0\) then we have \(F(g; g,v) = 0\) and hence
	\(\hat{y}=g\) is the
	minimizer of~\eqref{eq:ReducedProblemWithoutShift}.
	So we may assume \(\norm{gv} \neq 0\) in the following.
	We now distinguish whether \(\norm{yv}\neq 0\) or \(\norm{yv} = 0\). 
	In the first case, we may differentiate $F$,
	and setting the gradient of $F$ to zero results in
	\[
	0 = y-g + \frac{\nu}{\norm{yv}}(yv)v^\tT.
	\]
	We multiply by \(v\) to
	obtain
	\(
	yv-gv
	= -\nu\frac{yv}{\norm{yv}}\norm{v}^2.
	\)
	Rearranging yields
	\(
	\bigl(1+\nu\frac{\norm{v}^2}{\norm{yv}}\bigr)yv = gv,
	\)
	which implies \(\frac{yv}{\norm{yv}} = \frac{gv}{\norm{gv}}\), i.e., both
	vectors have the same direction.\\This leads to
	\[
	y = g - \nu\frac{gv}{\norm{gv}}v^\tT.
	\]
	For \(\norm{yv}=0\) we look at the subgradient of $F.$
	As condition for a minimizer $\hat{y}$, we have that $\hat{y}-g$ is in the subgradient of $\nu\norm{\hat{y}v}.$ For $y$ with \(\norm{yv}=0\), this subgradient is
	given as $\{zv^\tT\colon\norm{z} \leq \nu \}.$
	When considering the functional $F,$ the amplitude $m$ from the assertion of the lemma reads
	$m= \min\{\nu,\norm{gv}/\norm{v}\}$. If $\norm{gv}/\norm{v} < \nu,$ then $F$ is differentiable at $y=g-msv^\tT,$  $\|yv\| \neq 0,$ and we are in the previously considered case.
	Hence, we may assume that $m = \nu.$ Then, for $\hat{y}=g-msv^\tT,$ we have $\hat{y}-g= msv^\tT $ with $m = \nu$. This shows that $\hat{y}$ fulfills the condition of a minimizer.
	In consequence,~\eqref{eq:minim-E} is true for the functional $F$. 
	Then resubstituting shows~\eqref{eq:minim-E} for $E$,
	and plugging $\hat x$ into $E$ we get~\eqref{eq:E-Minimum}.

	It remains to show the implication~\eqref{eq:CompareDifferentfanda}.
	To this end, let $\mu = \frac{fw - a}{\norm{w}}$ and
	$\tilde \mu \coloneqq \frac{\tilde{f}w - \tilde a}{\norm{w}}$.
	By the assumption of~\eqref{eq:CompareDifferentfanda}, $\norm{\mu} < \norm{\tilde \mu}.$
   We consider three cases. If  $\norm{\tilde \mu} \leq \lambda$, then  $\norm{\mu} < \lambda.$
   Hence, the minimizer of $E(x;f,a,w)$ equals $\frac{1}{2} \|w\|_2^2 \norm{\mu}^2,$ and the one of $E(x;\tilde f,\tilde a,w)$ equals  $\frac{1}{2} \|w\|_2^2 \norm{\tilde \mu}^2 > \frac{1}{2} \|w\|_2^2 \norm{\mu}^2$
   which shows~\eqref{eq:CompareDifferentfanda}. If $\norm{\tilde \mu} > \lambda$ and 
   $\norm{\mu} < \lambda,$ we have to consider the second line of~\eqref{eq:E-Minimum} for 
   $\norm{\tilde \mu}$. From this we obtain a minimal value of~$E(x;\tilde f,\tilde a,w).$
   We have to show that   
   \[
   \norm{w}^2 \lambda \bigl(\frac{1}{2}\lambda + (\norm[\big]{\tilde \mu}  - \lambda)  ) > 
   \frac{1}{2} \norm{w}^2 \norm{\mu}^2;
   \] 
   but this is a consequence of the second summand on the left hand side being positive  
   and $\lambda \geq \norm{\mu}.$
   Finally, if both 
   $\norm{\tilde \mu} > \lambda$ and  $\norm{\mu} > \lambda,$
   we apply the second line of~\eqref{eq:E-Minimum} and see that we need  
   $\norm{\tilde \mu} - \lambda > \norm{\mu}  - \lambda $ for the statement to hold. 
   This is true by assumption which completes the proof.
\end{proof}

\begin{example}\label{ex:diff2-2d-prox}
	We continue with the situation from Example~\ref{ex:diff2-2d} and take three
	points~\(f^{(j)}\in\mathbb R^2\), \(j=1,2,3\) and
	denote~\(f = (f^{(1)}, f^{(2)} , f^{(3)})\).
	Depending on the chosen value for \(\lambda\)
	in the proximal mapping, there are two possibilities: If
	\(m=\frac{\norm{fw}}{\norm{w}^2} = \frac{\norm{fb_2}}{6} \leq \lambda\), we
	obtain three points \(x = \prox_{\lambda D_2}(f) = f-msb_2^\tT\) that
	lie on a line, cf. Figure~\ref{subfig:prox:line}. If \(m>\lambda\), then the
	result \(x\) of the proximal mapping does not yield \(D_2(x) = 0\),
	but the `movement' of the points in direction \(s\) is
	restricted by~\(\lambda b_2^\tT\), cf. Figure~\ref{subfig:prox:lambda}.
\end{example}
\begin{figure*}[tbp]\centering
	\begin{subfigure}[t]{.48\textwidth}\centering
		\includegraphics{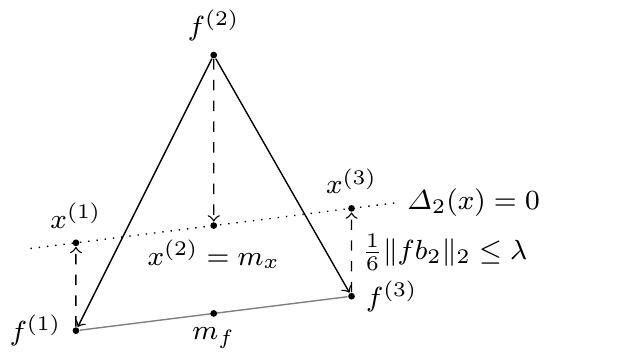}
		\caption{Minimizing \(D_2(f) = 2\norm{m_f-f^{(2)}}\) to\\
			zero by moving \(f\) to \(x\), where \(x^{(2)} = m_x\).}
		\label{subfig:prox:line}
	\end{subfigure}
	\begin{subfigure}[t]{.48\textwidth}\centering
		\includegraphics{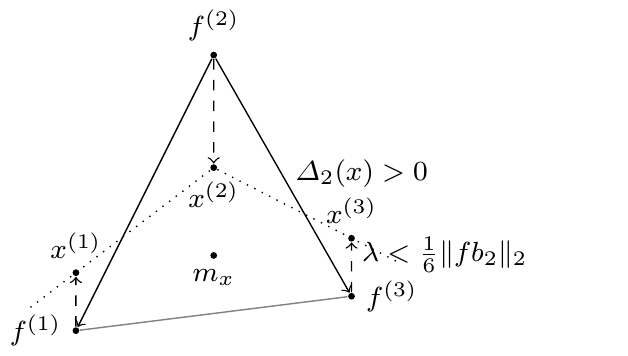}
		\caption{Tradeoff between minimizing \(D_2\) and a maximal movement
			of~
			\(f^{(j)}\), \(j=1,2,3\).
			}
		\label{subfig:prox:lambda}
	\end{subfigure}
	\caption{Two cases of the second order
		difference~\(D_2(f)\) in \(\mathbb R^n\) by looking at the plane
		generated by \(f^{(1)},f^{(2)},f^{(3)}\in\mathbb R^n\): The original points~\(f\) are
		moved onto~\(x\) towards forming a line:
		(\subref{subfig:prox:line})~yielding~\(D_2(x)=0\),  i.e. \(x_2\) is the mid point \(m_x\) of \(x_1\) and \(x_3\), and~(\subref{subfig:prox:lambda}) reducing \(D_2(x)<D_2(f)\)
		but restricting the movement to be less then~\(\lambda\in\mathbb R\)
		for~\(f^{(1)},f^{(3)}\) and~\(2\lambda\) for \(f^{(2)}\) respectively.
		}
	\label{fig:2Ddiff2}
\end{figure*}
%

After these preparations, we now deal with the proximal mappings needed for the inpainting problems~\eqref{eq:Functional4InpaintingNoiselessCombCyclicVec}
and~\eqref{eq:Functional4InpaintingNoisyCombCyclicVec}
for combined cyclic and vector space data.
In particular, each data item now is an element of \(\mathcal X = (\mathbb S^1)^m \times \mathbb R^n\).
As motivation, let us first have a look at the first order difference $D_1$ and 
the inpainting problem~\eqref{eq:Functional4InpaintingNoiselessCombCyclicVec} for noiseless data.
By the constraint $x_{ij}= f_{ij}$ outside the inpainting region, it might happen that at the boundary 
of the inpainting region the member $x_{ij}= f_{ij}$ is fixed but its neigbor, say $x_{i,j+1},$ may vary.
Then, we have to study the corresponding functional $D_1(x_{ij},x_{i,j+1})$ for fixed $x_{i,j}$ and find its
proximal mapping. The following theorem deals with this issue in a more general setup.

Before we state the theorem we introduce some notation: By again using a representation system \(\mathbb S^1 \cong[-\pi,\pi)\), we can interpret any \(x\in\mathcal X\) as a vector~\(x = (x_{\mathbb S}, x_{\mathbb R})^\tT \in \mathcal X^d\), where the same representation system is used for all cyclic components.\\
We define \((\cdot)_{\mathcal X}\colon\mathbb R^{(m+n)\times d} \to \mathcal X^d\) by
\begin{equation}\label{eq:DefBrakRoundX}
(x)_{\mathcal X} \coloneqq ((x_{\mathbb S})_{2\pi}, x_{\mathbb R})^\tT, 
\end{equation}
where \((x_{\mathbb S})_{2\pi}\) is defined as the component-wise application of \((\cdot)_{2\pi}\) as given in~\eqref{eq:DefRoundBracket}, i.e. is applied to
the \(m\) cyclic components of each column vector \(x^{(i)}\in\mathcal X\).
Similarly we define
\begin{equation}\label{eq:DefBrakX}
[x]_{\mathcal X} \coloneqq ([x_{\mathbb S}]_{2\pi}, x_{\mathbb R})^\tT, 
\end{equation}
where \([\cdot]_{2\pi}\) defined in~\eqref{eq:DefMultivaluedBracket} is applied to the phase-valued components of \(x\).

In the following we consider a weight vector $w \in \mathbb R^d$, a data matrix
$x = (x^{(1)},\ldots,x^{(d)})$ with each member \(x^{(i)}\) having values in
\(\mathcal{X} = (\mathbb S^1)^m \times \mathbb R^n\)
and a subset $A \subset \{1,\ldots,d\}.$
We partition $w$ into a variable part $w^a$ and into a fixed part $\tilde{w}$ according to whether the index $i$ of $w_i$
belongs to $A$ or not.
Accordingly, we partition $x$ into a variable part $x^a$ and into a fixed part $\tilde{x}$ 
and consider the mappings
\begin{equation}\label{eq:DifferencesWithFixedComponents}
\varphi_A\colon x^a  \mapsto D(x,w)
\end{equation}
for the corresponding differences $D(x,w)$, where only the $x^a$ are considered as variable and the \(\tilde x\) are fixed values.
We derive an explicit representation for the corresponding proximal mappings in the following theorem.
 
\begin{theorem} \label{thm:proxy_w}
	Let $w$ be one of the weights 
	$w=(-1,1),$ $w=(1,-2,1),$ or $w=(-1,1,1,-1)$ 
	which corresponds to considering the first order difference $D_1$
	and the second order differences $D_2$ and $D_{1,1}$, respectively.
	Let $d$ be the respective length of $w,$ $A \subset \{1,\ldots,d\},$
	and $w$ be partitioned into the corresponding variable part $w^a$ and into a fixed part $\tilde{w}$.
	We partition $x,f \in \mathcal{X}^d$ accordingly and let $\tilde{f} = \tilde{x}.$
	Then, the proximal mapping of $\varphi_A$ defined in~\eqref{eq:DifferencesWithFixedComponents}
	is given by 
	\begin{equation}\label{prox4combinedWithFixed}
		\prox_{\lambda \varphi_A}(f^a) =  (f^a - m \, s \,(w^a)^\tT)_{\mathcal X},
	\end{equation}	
	with the parameter $\lambda > 0;$ here, the direction(s) $s$ and amplitude $m$ are given by
	\begin{align}
		\label{eq:form_s_and_m}
	s &=
	\begin{cases}
	    \frac{[fw]_{\mathcal X}}{\norm{(fw)_{\mathcal X}}}
	         &\mbox{ if \( \norm{(fw)_{\mathcal X}}\neq 0\),}\\
	      0  &\mbox{ else,}
	\end{cases}\nonumber
	\intertext{and}
	m &= \min\biggl\{\lambda, \frac{\norm{(fw)_{\mathcal X}}}{\norm{w^a}^2}\biggr\}.
	\end{align}
\end{theorem}
\begin{remark-Non}
	We note that the bracket $[\cdot]_{\mathcal X}$ and thus the proximal
	mapping (having an additional value for each additional instance of $s$), is
	multivalued if some components of \((f_{\mathbb S}w)_{2\pi}\) are equal
	to~\(-\pi\).
	More precisely, if there are \(l\in\{1,\ldots,n\}\) such components,
	we obtain \(2^l\) solutions from the different instances the vector \(s\)
	might take.
	The reason for this is, that the mapping $\varphi_A$ is no longer
	convex for data in $\mathcal X,$ and that the minimizer defining the proximal
	mapping is not necessarily unique.
	Owing to this observation, we consider set valued proximal mappings
	gathering all minimizers.
	We notice that the above proximal mapping is single-valued
	if and only if \((f_{\mathbb S}w)_{2\pi}\in(-\pi,\pi)^d\). This is the
	generic case.
	The degenerate case involving antipodal points appears rather seldom in
	practice; at least, in a non-artificial noisy  setup, it is very unlikely to
	encounter antipodal points. Then, at least the proximal mapping is single-valued yielding a deterministic result.
		However, we notice that jumps of hight close to $\pi$ are problematic since then stability issues appear. 
	Furthermore, data with antipodal points or almost antipodal points may often
	be interpreted as not fine enough sampled data. This means, if the sampling
	rate is higher, the distance of nearby data items might get smaller and the
	situation might just disappear. We note that this does not exclude the possibility of jumps in the finer sampled data.
	The point is that large critical jumps might be revealed as smaller jumps.
\end{remark-Non}
\begin{proof}[Proof of Theorem~\ref{thm:proxy_w}.]
	 In order to derive explicit formulae for the proximal mappings, we have to find the 
	 minimizer(s) of
	 \begin{equation}\label{eq:Xdfunctional}
	    \mathcal E_{\mathcal X} (x^a; f^a,w) \coloneqq
	    \frac{1}{2}\sum_{j \in A} d_{\mathcal X}(x^{(j)},f^{(j)})^2 + \lambda D(x;w).
	 \end{equation}
	 By Lemma~\ref{lem:RepresentationCombinedDifferences}
	 we may rewrite $E_{\mathcal X} (x^a; f^a,w)$ as 
	\begin{align}
	    \mathcal E_{\mathcal X}&(x^a; f^a,w)\\
	       &= 	\frac{1}{2}\sum_{j \in A}
	       \norm[\big]{f_{\mathbb R}^{(j)}-x_{\mathbb R}^{(j)}}^2
	\\&\qquad + \frac{1}{2}\sum_{j \in A}
	\min_{k^{(j)}\in\mathbb Z^{m}}
	\norm[\big]{f^{(j)}_{\mathbb S} - x_{\mathbb S}^{(j)} - 2\pi k^{(j)}}^2
	\\&\qquad	+
	\min_{\sigma\in\mathbb Z^m}
	\lambda\sqrt{
		\norm{x_{\mathbb R}w}^2+\norm{x_{\mathbb S}w-2\pi\sigma}^2}.
	\end{align}
We now use that $\tilde{x}=\tilde{f}$ and employ the notation 
\(\norm{\cdot}[\text{F}]\) for the Frobenius norm. For the remaining part of the proof, let \(t \coloneqq |A|\). We obtain that
  \begin{equation}\label{eq:IntroEks}
	  \begin{split}
  &\mathcal E_{\mathcal X}(x^a; f^a,w)\\
  &\quad=       
   \min_{\substack{k\in\mathbb Z^{m\times t}\\\sigma\in\mathbb Z^m}} 
  	    \Bigg(
  	    \frac{1}{2}  \norm{f^a_{\mathbb R} - x^a_{\mathbb R}}[F]^2 + 
        \frac{1}{2}  \norm{f^a_{\mathbb S} - x^a_{\mathbb S} - 2 \pi k}[F]^2\\
        &\qquad+\lambda   \sqrt{  
        	\norm{x^a_{\mathbb R}w^a + \tilde f_{\mathbb R} \tilde w}^2+
        	\norm{x^a_{\mathbb S}w^a - (2\pi\sigma - \tilde f_{\mathbb S} \tilde w) }^2
  	    }  
  	    \Bigg)\!.
  	  \end{split}
  \end{equation}
Using this, we rewrite the minimization problem to
\begin{align} 
		&\min_{x^a\in[-\pi,\pi)^{m \times t} \times \mathbb R^{n \times t}} \mathcal E_{\mathcal X} (x^a; f^a,w)
	\\
	&\quad=
    \min_{\substack{k\in\mathbb Z^{m \times t}\\\sigma\in\mathbb Z^m}}
	 \ 
    \min_{x^a\in[-\pi,\pi)^{m \times t} \times \mathbb R^{n \times t}} E_{k,\sigma}(x^a; f,w) \notag
	 \\
    &\quad=
    \min_{\substack{k\in\mathbb Z^{m \times t}\\\sigma\in\mathbb Z^m}}
	 \ \ 
    \min_{x^a\in[-\pi,\pi]^{m \times t} \times \mathbb R^{n \times t}} E_{k,\sigma}(x^a; f,w),
    \label{eq:ReducedProblemInProxTheoProof}
\end{align}
where
\begin{equation}\label{eq:DefEks}
	\begin{split}
		&
		E_{k,\sigma}(x^a; f,w)
		\\
		\quad
		&=
		\frac{1}{2}  \lVert f^a_{\mathbb R} - x^a_{\mathbb R}\rVert^2_F +
 		\frac{1}{2}  \lVert f^a_{\mathbb S} - x^a_{\mathbb S} - 2 \pi k\rVert^2_F\\
&\quad+\lambda\sqrt{
	\norm{x^a_{\mathbb R}w^a + \tilde f_{\mathbb R} \tilde w}^2+
	\norm{x^a_{\mathbb S}w^a - (2\pi\sigma - \tilde f_{\mathbb S} \tilde w) }^2
}.
\end{split}
\end{equation}
Having a look at~\eqref{eq:DefEks}, there exist values~$\tilde{k},\tilde{\sigma}$ such that
$E_{\hat{k},\hat{\sigma}}(\hat{x})= E_{\tilde{k},\tilde{\sigma}}(\tilde{x}).$ 
Summing up, the problem reduces to finding the minimizers of all $E_{k,\sigma}$ in $[-\pi,\pi]^{m \times t}$
and comparing their value.
For the remaining part of the proof, let \(0_n\) be a zero (column) vector of length \(n\) and \(0_{n,t}\) be a zero-matrix of dimension \(n\times t\).

For any $k,\sigma,$ the functional $E_{k,\sigma}$ has a unique minimizer
given by Lemma~\ref{lem:proxRn} as
\begin{align}\label{eq:temp1}
\hat{x}^a_{k,\sigma} &= f^a-2\pi 
\begin{pmatrix}
	k\\0_{n,t}
\end{pmatrix}
- s_{k,\sigma}m_{k,\sigma}(w^a)^\tT.
\end{align}
We first derive $s_{k,\sigma}$ using Lemma~\ref{lem:proxRn}, where we use the notation,
\begin{align}
 s_{k,\sigma}= \nu_{k,\sigma}/ \|{\nu_{k,\sigma}}\|_2.
\end{align}
We notice that the data for Lemma~\ref{lem:proxRn} is given by $f^a-2\pi
\left(\begin{smallmatrix} k\\0_{n,t}\end{smallmatrix}\right)
$
and the offset $a$ in the same lemma is $a= 2\pi
\left(\begin{smallmatrix} \sigma\\0_{n}\end{smallmatrix}\right)
- \tilde{f}\tilde{w}$.
We get
\begin{equation} \label{eq:defNuksigma}
	\begin{split}
		\nu_{k,\sigma} &= \left(f^a-2\pi
		\left(\begin{smallmatrix} k\\0_{n,t}\end{smallmatrix}\right)
		\right)w^a + \tilde{f}\tilde{w} - 2 \pi
		\left(\begin{smallmatrix} \sigma\\0_{n}\end{smallmatrix}\right)\\
		&= fw - 2\pi \left(
		\left(\begin{smallmatrix} k\\0_{n,t}\end{smallmatrix}\right)
		w^a + 
	\left(\begin{smallmatrix} \sigma\\0_{n}\end{smallmatrix}\right)
	\right),
	\end{split}
\end{equation}
and
\begin{align}
m_{k,\sigma} = \min\Biggl\{
\lambda, \frac{\nu_{k,\sigma}}{\norm{w^a}^2}
\Biggr\}.
\end{align}
By~\eqref{eq:CompareDifferentfanda}, we have to find the minimum of the norms $\norm{\nu_{k,\sigma}}$ with respect to $k,\sigma$ in order to find the minimum or minima and their corresponding minimizers of the $E_{k,\sigma}$.
We first consider the following special case, where
\begin{align}\label{eq:special-case}
(fw)_i \notin 2\pi \Z - \pi  \quad \text{for all} \quad  i \in \{1,\ldots,m\}.
\end{align}
More precisely, in none of the cyclic data dimensions \(i=1,\ldots,m\) of \(f\in\mathcal X^{d}\), the scalar product with the weight \((fw)_i\) is an odd multiple of \(\pi\).
There exist $r_1,\ldots,r_m\in\mathbb Z$ such that $(fw)_i - 2\pi r_i \in (-\pi,\pi)$ for all $i \in \{1,\ldots,m\}$.
Then $\norm{\nu_{k,\sigma}}$ is minimal with respect to $k,\sigma$ if and only if 
$kw^a + \sigma = r$ where $r=(r_1,\ldots,r_m)^\tT$, cf.~\eqref{eq:defNuksigma}.
For each fixed matrix $\hat{k}\in\mathbb Z^{m\times t}$ there is a uniquely determined vector $\hat\sigma = \hat\sigma(\hat k)$ solving this system of linear equations.
Such a pair minimizes $\norm{\nu_{k,\sigma}}$ w.r.t. $k,\sigma$ and 
\begin{align}
\nu_{\hat{k},\hat{\sigma}
} = (fw)_{\mathcal X}.
\end{align}
Using~\eqref{eq:temp1} we obtain the corresponding minimizer w.r.t.\ \(x\) as
\begin{align}\label{eq:hatXpartCase}
\hat{x}^a_{\hat{k},\hat{\sigma}
} 
 &= f^a-2\pi
\begin{pmatrix}
	\hat k\\0_{n,t}
\end{pmatrix}
- ms(w^a)^\tT,
\end{align}
with $m,s$ as given in~\eqref{eq:form_s_and_m}.
Now there is precisely one $k^\ast$ with its corresponding \(\hat\sigma = \hat\sigma(k^\ast)\) such that $\hat{x}^a_{k^\ast,\hat{\sigma}
} \in [-\pi,\pi)^{m\times t}\times\mathbb R^{n\times t}$ which implies that
\begin{align}
x^\ast = \hat{x}^a_{k^\ast,\hat{\sigma}} = (f^a - ms(w^a)^\tT)_{\mathcal X}
\end{align}
is a minimizer of $\mathcal E_{\mathcal X}$
by~\eqref{eq:ReducedProblemInProxTheoProof}.
Concerning uniqueness we notice that $\sigma \mapsto \norm{\nu_{k,\sigma}}$ has
a unique minimizer.
By~\eqref{eq:CompareDifferentfanda} this implies that one may minimize w.r.t.\ $k,x$ in~\eqref{eq:ReducedProblemInProxTheoProof} choosing $\sigma = \hat{\sigma(k)}$ as previously in this proof.
This unique minimizer is $x^\ast$ since it is the only candidate with its first \(m\) components in $[-\pi,\pi)$. This finishes the special case~\eqref{eq:special-case}.

For the general case, let $G\subset\{1,\ldots,m\}$ be the set such that
\begin{align}
(fw)_i \in 2\pi \Z - \pi  \quad \text{for }\quad i\in G 
\end{align}
If $G$ is empty, we are in the case~\eqref{eq:special-case} we considered. Furthermore let \(G^C \coloneqq \{1,\ldots,m\}\backslash G\).
Hence for each $i \in G^C$ we find $r_i$ such that $(fw)_i - 2\pi r_i \in (-\pi,\pi)$ with the same arguments as for the case~\eqref{eq:special-case}.
For each $i \in G$ there is $r_i \in \Z$ such that
$(fw)_i= (2r_i-1)\pi$. 
Then $\nu_{k,\sigma}$ attains its smallest value exactly when
$k_i w^a + \sigma_i = r_i,$ for all $i \in G^C$
and when $k_i w^a + \sigma_i \in \{r_i,r_i-1\},$ for all $i\in G$.
Following the same steps as above, let us fix $\hat{k}$ and determine all $\hat{\sigma}_u = \hat\sigma(\hat{k},u)$, $u \in \{0,1\}^m, u_i=0$ for \(i\in G^C\), that fulfill one of the systems of equations
\begin{align}
\sigma(\hat{k},u) = r-\hat{k}w^a - u, \quad\text{where } r=(r_1,\ldots,r_m)^\tT
.
\end{align}
These are $2^{|G|}$ systems systems of equations, one for each value of \(u\), each having a unique solution.
By~\eqref{eq:CompareDifferentfanda}, we only  have to consider the functionals $E_{k,\sigma}$ corresponding
to such a pair of parameters $(\hat{k},\hat\sigma_u)$.
We get that the components of $\nu_{\hat k, \hat\sigma_u}$ with indices in $G^C$ are given as in~\eqref{eq:defNuksigma},
whereas the components with indices in $G$ are given by 
\begin{equation*}
(\nu_{\hat{k},\hat\sigma_u})_i = (-1)^{u_{i}+1}\pi.	
\end{equation*}
From $\nu_{\hat{k},\hat\sigma_u},$ we obtain $s_{\hat{k},\hat\sigma_u}$ again by using~\eqref{eq:form_s_and_m} for this special \(\nu\).
We furthermore notice that $m_{\hat{k},\hat\sigma_u}$ equals the $m$ given in~\eqref{eq:form_s_and_m} and is especially independent of the particular choice of $u$. Hence we get  
\begin{equation*}
	\hat x^a_{\hat{k},\hat\sigma_u}
	= f^a-2\pi
	\begin{pmatrix}
		\hat k\\0_{n,t}
	\end{pmatrix}
	 - s_{\hat{k},\hat\sigma_u} m(w^a)^\tT.
\end{equation*}
Now we follow the lines of case~\eqref{eq:special-case} to conclude that 
$ \hat x^a = (f^a - s \, m \,w^a)_{\mathcal X}$, where \(s\) is one of the $2^{|G|}$ instances of $\frac{[fw]_{\mathcal X}}{\norm{(fw)_{\mathcal X}}}$, and these are precisely the minimizers of ${\mathcal E}(x^a;f^a,w)$ w.r.t. \(x^a\) lying in $[-\pi,\pi)^{m \times t} \times \mathbb R^{n \times t}$. This completes the proof.
\end{proof}
%
%
\begin{figure*}[tbp]\centering
	\begin{subfigure}[t]{.48\textwidth}\centering
		\includegraphics{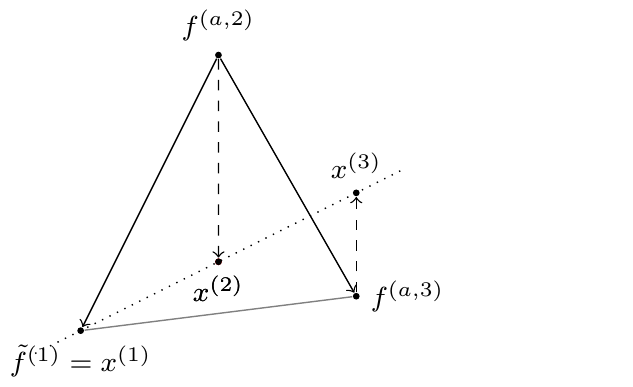}
		\caption{Minimizing \(\tfrac{1}{5}D_2 \leq \lambda\) with fixed \(f^{(1)}\).}
		\label{subfig:prox:fixed}
	\end{subfigure}
	\begin{subfigure}[t]{.48\textwidth}\centering
		\includegraphics{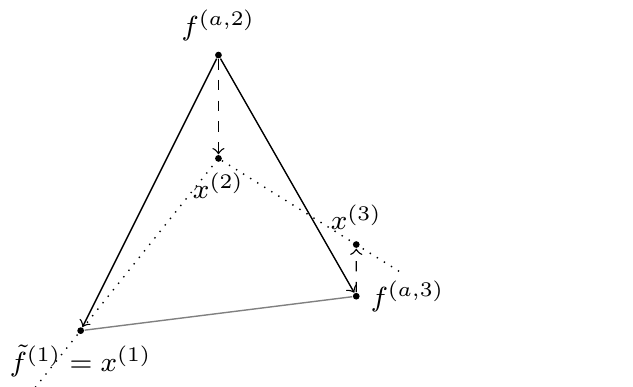}
		\caption{Minimizing \(\tfrac{1}{5}D_2 > \lambda\) with fixed \(f^{(1)}\).}
		\label{subfig:prox:fixed-2}
	\end{subfigure}
	\caption{Minimizing a second order difference in \(\mathbb R^2\), where \(f^{(1)}\) is fixed, i.e.\ \(x^{(1)} = f^{(1)}\) and \(A=\{2,3\}\) are the active data points. Again, two different values of \(\lambda\) are shown: (\subref{subfig:prox:fixed}) \(\lambda > \tfrac{1}{5}D_2(f;w)\), i.e.\ the corresponding proximal mapping reaches the minimum, (\subref{subfig:prox:fixed-2}) \(\lambda < \tfrac{1}{5}D_2(f;w)\), i.e.\ the corresponding proximal mapping is just a step reducing the value \(D_2(x;w) < D_2(f;w)\).}\label{fig:diff2-2d-fixed}
\end{figure*}
\begin{example}\label{ex:diff2-2d-fixed}
	We continue with the situation from Examples~\ref{ex:diff2-2d} and~\ref{ex:diff2-2d-prox}, but split the three points \(f=(f^{(1)}, f^{(2)}, f^{(3)})\) into an active and fixed part using \(A=\{2,3\}\). In other words \(\tilde f^{(1)}\) is fixed, we have \(x^{(1)} = \tilde f^{(1)}\), and both \(f^{(a,2)}\) and \(f^{(a,3)}\) are active, i.e.\ \(x^{(2)},x^{(3)}\) are affected by the restricted proximal map. The movement again depends on the chosen value for \(\lambda\) which is also illustrated in Figure~\ref{fig:diff2-2d-fixed}.
\end{example}

We also need proximal mappings for the data term \(F(x; f)\). 
\begin{proposition}\label{prop:data-prox}
	For \(f,g \in \mathcal X^k\), \(k=NM\), let
	\[
	\mathcal E (x; g,f)
	=
	\sum_{i=1}^k d_{\mathcal X}(g^{(i)},x^{(i)})^2 + \lambda d_{\mathcal X}(f^{(i)},x^{(i)})^2.
	\]
	Then, the minimizer(s) \(\hat x\) of \(\mathcal E\) are given by
	\[
	\hat x = \Bigl( \frac{g+\lambda f}{1+\lambda} + \frac{\lambda}{1+\lambda}2\pi v\Bigr)_{\mathcal X},
	\]
	where \(v = \bigl(v^{(i)}_{j}\bigr)_{i,j=1}^{k,n+m}\in\mathcal X^k\) is defined by
	\[
	v^{(i)}_j = \begin{cases}
	0 &\mbox{ if } j>m,\\
	0 &\mbox{ if } \abs{g_j^{(i)} - f_j^{(i)}} \leq \pi, j\leq m,\\
	\sgn(g_j^{(i)} - f_j^{(i)}) &\mbox{ if } \abs{g_j^{(i)} - f_j^{(i)}} > \pi, j\leq m.
	\end{cases}
	\]
\end{proposition}
\begin{proof}
We observe that the functional \(\mathcal E\) under consideration only involves squared distances.
This can be handled separately for every index \((i,j)\in\{1,\ldots,k\}\times\{1,\ldots,n+m\}\). Hence, the proposition follows from considering both the $\mathbb R$-valued data and the $\mathbb S^1$-valued data case separately. This has been done in \cite{bergmann2014second} in Lemma 3.3 and Proposition 3.7, respectively.
\end{proof}
%
%
%
\subsection{Cyclic proximal point algorithms}
\label{subsec:InpaintDenoiseCPPA}
%
%
%
The proximal mappings from Theorem~\ref{thm:proxy_w} can be efficiently
applied in parallel to compute minimizers of~%
\eqref{eq:Functional4InpaintingNoiselessCombCyclicVec} and~%
\eqref{eq:Functional4InpaintingNoisyCombCyclicVec} using a cycle length in the proximal point algorithm from Section~\ref{subsec:cppaGeneralExplanation} by splitting the functionals accordingly. 

\paragraph{A splitting for noiseless inpainting.}\par
Each summand in the first and second order differences in~%
\eqref{eq:Functional4InpaintingNoiselessCombCyclicVec} is incorporated into a proximal mapping using Theorem~\ref{thm:proxy_w} by setting the values affected by the subjection as fixed and keeping the remaining ones as active. 

If two summands act on distinct data, their proximal mappings can be computed in parallel. This reduces the cycle length \(c\) of the CPPA tremendously and provides an efficient, parallel implementation. In the following, we will split each of the summands
\begin{equation} \label{eq:splittingJ}
J(x) = \alpha \operatorname{TV}_{1}(x)
	+ \beta \operatorname{TV}_{2}(x)
	+ \gamma \operatorname{TV}_{1,1}(x) 
\end{equation}
into 
\begin{equation} \label{splitting_18}
J = \sum\nolimits_{l=1}^{18} J_l
\end{equation}
where the summands for the first order horizontal difference are split as
\begin{equation*}
	\begin{split}
		J_1+J_2 &\coloneqq\alpha_1 \sum_{(i,j)} D_1(x_{2i,j},x_{2i+1,j})
		\\&\quad +
		\alpha_1 \sum_{(i,j)} D_1(x_{2i+1,j},x_{2i+2,j}),
	\end{split}
\end{equation*}
which similarly yields \(J_3,J_4\) for the vertical first order difference and \(J_5,\ldots,J_8\) for the two diagonal sums.
The second order term is decomposed into three sums \(J_9\), \(J_{10}\), \(J_{11}\), given by 
\begin{align*}
   J_9   &\coloneqq \beta_1 \sum_{(i,j)} D_2(x_{3i-1,j},x_{3i,j},x_{3i+1,j}),\\
   J_{10}  &\coloneqq \beta_1 \sum_{(i,j)} D_2(x_{3i,j},x_{3i+1,j},x_{3i+2,j}),\\
   J_{11}  &\coloneqq \beta_1 \sum_{(i,j)} D_2(x_{3i+1,j},x_{3i+2,j},x_{3i+3,j}),
\end{align*}   
and analogously we obtain $J_{12},\ldots,J_{14}$ for the vertical second order difference. Splitting the term $\gamma \operatorname{TV}_{1,1}(x)$ in this manner yields $J_{15},\ldots,J_{18}$ as follows:
\begin{align*}
   J_{15}   &= \gamma \sum_{(i,j)} D_{1,1}(x_{2i,2j},x_{2i+1,2j},x_{2i,2j+1},x_{2i+1,2j+1}),\\
   J_{16}   &= \gamma \sum_{(i,j)} D_{1,1}(x_{2i+1,2j},x_{2i+2,2j},x_{2i+1,2j+1},x_{2i+2,2j+1}),\\
   J_{17}   &= \gamma \sum_{(i,j)} D_{1,1}(x_{2i,2j+1},x_{2i+1,2j+1},x_{2i,2j+2},x_{2i+1,2j+2}),\\
	&\begin{aligned}
		   \mathllap{J_{18}}= \gamma \sum_{(i,j)} D_{1,1}(x_{2i+1,2j+1},x_{2i+2,2j+1},\\[-.5\baselineskip]\hspace{10em} x_{2i+1,2j+2},x_{2i+2,2j+2}).
	\end{aligned}
\end{align*}
In any of these functions \(J_l, l\in\{1,\ldots,18\}\), any data point \(x_{i,j}\in\mathcal X\) is occurring at most once and hence all proximal maps a function \(J_l\) consists of can be evaluated in parallel. This leads to a cycle length of~\(c=18\).

\paragraph{A splitting for combined inpainting and denoising.}\par
In order to derive a cyclic proximal point algorithm for the combined inpainting and denoising
model~\eqref{eq:Functional4InpaintingNoiselessCombCyclicVec}, we encounter two
differences compared to the previous derivation: All data is always marked as
active because no index \((i,j)\in\Omega_0\) is restricted by a constraint and we further obtain a data term, i.e. we additionally have
\begin{equation}\label{eq:splitting-dataterm2DX}
	J_{19} = \sum\nolimits_{(i,j)\in\Omega^C} d_{\mathcal X}(x_{i,j},f_{i,j})^2
\end{equation}\enlargethispage{\baselineskip}
which can be evaluated in parallel using the proximal mapping given by Proposition~\ref{prop:data-prox}.

\paragraph{Initialization.}\par
In order to initialize the algorithm, we employ an the idea of unknown boundary used in~\cite{AF13},
which can easily be implemented during the first iterations of the CPPA: all unknown values \(x_{i,j}, (i,j) \in\Omega\) 
are initialized whenever setting a corresponding difference
\(D(x,b)=0\) yields a unique value for this pixel. Afterwards, 
this data item is set to be known and can be used to initialize other unknown pixel. Hence after at most~\(k=\max\{N,M\}\)
 iterations, all pixels are known.

The complete procedure for both models of noiseless and noisy inpainting is summarized in Algorithm~\ref{alg:CPPA}.
\begin{algorithm*}[tbp]
	\caption[]{\label{alg:CPPA} CPPA for minimizing~\eqref{eq:2DTVfunctional} for \((\mathbb S^1)^n\times R^m\) valued data}
	\begin{algorithmic}
		\State \textbf{Input} a sequence $\{ \lambda_k \}_k$ of positive values, cf.~\eqref{eq:CPPAlambda},
		\State parameters $\alpha = (\alpha_1,\alpha_2,\alpha_3,\alpha_4)$, $\beta = (\beta_1,\beta_2)$, $\gamma$,
		\State a set \(\Omega\subset\Omega_0\), and
		data \(f\in\mathcal X^{N\times M}\)\\\vspace{-.5\baselineskip}
		\Function {CPPA}{$\alpha$, $\beta$, $\gamma$, $\{ \lambda_k \}_k$, $f$}
		\State Initialize \(x_{i,j}^{(0)}=f_{i,j},\ (i,j)\in\Omega^C\), \(x_{i,j}^{(0)}\) as active, as unknown for \((i,j)\in\Omega\) and \(k=0\)
		\State Initialize the cycle length as \(c=18\) (noiseless case) or \(c=19\) (noisy case)
		\Repeat
		\For{$l \gets 1$ \textbf{to} $c$}
			\State \(x^{(k+\frac{l}{c})} \leftarrow \prox_{\lambda_k J_l}(x^{(k+\frac{l-1}{c})})\)
			\State employing Theorem~\ref{thm:proxy_w} and Proposition~\ref{prop:data-prox}
		\EndFor
		\State $k \gets k+1$
		\Until a convergence criterion are reached
		\State\Return $x^{(k)}$
		\EndFunction
	\end{algorithmic}
\end{algorithm*}
\subsection{Convergence Analysis}
\label{subsec:Convergence}
%
%
%
As typical when dealing with nonlinear
geometries, we show convergence under certain conditions.
These conditions are comparable with the ones employed in~\cite{bergmann2014second}.
In particular, the data is assumed to be dense enough in the sense quantified later on.
This means that the data has to be locally nearby which does not mean that circular data components
are (globally) restricted to certain sectors -- the data in these components my wrap around.
Similar restrictions on the nearness of data and even more severe restrictions requiring almost equidistant-data    
have been imposed in the analysis of nonlinear subdivision schemes;
see, e.g., \cite{WallnerDynCAGD,harizanov2011normal,WeinmannConstrApprox}.
As it is also pointed out in these references, the analysis is qualitative in the sense that empirically convergence is observed for a  significantly wider range of input data.

Compared with the pure denoising setup considered in~\cite{bergmann2014second}, 
there are several issues we have to deal with in the
inpainting situation here: first, the proof there
relies on the uniqueness of the minimizers in the unwrapped situation which is
not given for inpainting;
second, a main step in the aforementioned proof is based on bounding the distances
\(d_{\mathcal X}(f_{i,j},x_{i,j})\) for all pixels \((i,j)\), to get information on $x$, 
while for inpainting the values \(f_{i,j}\), \((i,j)\in\Omega\), are missing.
In addition we consider a more general data space.

We first discuss the conditions we impose for our convergence analysis.
Then we derive the necessary information to prove our main result formulated as
Theorem~\ref{thm:Convergence}. It states that both the algorithm proposed for
inpainting and the algorithm proposed for simultaneous inpainting and
denoising converge to a minimizer.

We employ the following notation to denote the distance on the first $m$ components of two data items $x,y \in \mathcal X.$
We notice that those are the \(\mathbb S^1\)-valued components.
We let
\[
d_{\mathcal X,m}(x,y) \coloneqq d_{(\mathbb S^1)^m}(x_{\mathbb S},y_{\mathbb S}).
\]
Our first condition is that the data $f\colon\Omega^C  \to \mathcal X$ given on the
complement of the inpainting region \(\Omega\) is \emph{dense enough} in the
sense that the distance between pixels and their neighbors in $\Omega^C$ is
sufficiently small (in the respective spherical components).

In order to give a precise definition of this we require some preparation first.
On the domain grid, we consider the  distance $d((i,j),(k,l))$ which is the length of the shortest path 
with respect to the eight-neighborhood (with diagonal distance $\sqrt{2}$) connecting the indices $(i,j)$ and $(k,l).$
We consider a covering of the image domain $\Omega_0$ with balls $\{B((i,j),r_{i,j}):(i,j) \in \Omega^C\}$ centered at pixels
(i$,j) \in \Omega^C$ in the complement of the inpainting area. Ne note that the radius of the ball $B((i,j),r_{i,j})$  
may vary with $(i,j) \in \Omega^C$ and that balls are formed w.r.t.\ the  distance $d((i,j),(k,l))$ introduced right before.
We require that the graph induced by this covering is connected. Here the induced graph is formed by using the indices in 
$\Omega^C$ as vertices and by connecting to vertices $(i,j)$ and $(k,l)$ whenever $(k,l) \in B((i,j),r_{i,j})$	
or vice versa meaning that $(i,j) \in B((k,l),r_{k,l}).$
	 
Using such a covering $\{B((i,j),r_{i,j}): (i,j) \in \Omega^C\},$ as well as the shorter notation $N_{i,j}= B((i,j),r_{i,j}),$  we let  	 
	\begin{equation}\label{eq:defSupNormOfDiff}
	d^\Omega_\infty(f)  = \max_{(i,j) \in \Omega^C} %
	\max_{(k,l) \in \mathcal N_{i,j} \cap \Omega^C} d_{\mathcal X, m}(f_{i,j},f_{k,l}).
	\end{equation}			
	We emphasize that imposing a bound on~\eqref{eq:defSupNormOfDiff} still allows for jumps which are not too high. 
	As already pointed out, such restrictions on the nearness of data are typical for the analysis of algorithms of nonlinear data in general; cf. \cite{WallnerDynCAGD,harizanov2011normal,WeinmannConstrApprox,bergmann2014second}.
	We note that $d^\Omega_\infty(f)$ depends on the chosen covering and that we suppress this dependence in the 
	notation. The above definition takes the inpainting region into account and only restricts the spherical components.
	It turns out that for the non-spherical linear space components, no restrictions are necessary, and that large distances in these components do \emph{not} influence the behavior in the spherical components \emph{negatively}.	
	When there is no inpainting region, i.e. $\Omega = \emptyset$, and we are in  the pure denoising situation,
	we fix the covering by fixing $N_{i,j}$ to be the eight-neighbourhood of $(i,j)$
	and use the notation
	\begin{equation*}
		d_\infty(f) = d^{\,\emptyset}_\infty(f). 
	\end{equation*}
	Using this notation, we assume that (i) the quantity $d^\Omega_\infty(f)$ is small enough;
	precise bounds on $d^\Omega_\infty(f)$ are given in the lemmas and theorems later on.

Our second requirement is that (ii) the parameters~$\alpha, \beta, \gamma$ are
sufficiently small.
For large parameters, solutions become almost constant which
is often undesired and causes an interpretation problem, e.g.\ when the
original data is equally distributed around the circle.
Finally, we require (iii) that the parameter sequence $\{\lambda_k\}_k$ of the CPPA 
fulfills~\eqref{eq:CPPAlambda} with a small $\ell^2$ norm. The latter
can be achieved by rescaling the parameter sequence.
  
Our analysis is based on an unwrapping procedure which means that we `lift
the whole setup' to the universal covering of~%
$\mathcal X$ which we denote by $\mathcal Y = \mathbb R^{n+m}$.

	Similar to the notation $d_{\mathcal X,m},$ we use
	\[
	\lvert x- y\rvert_{\mathcal Y,m} \coloneqq \norm{(x_j)_{j=1}^m - (y_j)_{j=1}^m}
	\]
	to denote the distance on the first \(m\) components of the data $x,y \in \mathcal Y$.

Universal coverings stem from algebraic topology. We refer
to~\cite{hatcher2002algebraic} for an introduction. A covering consists of 
a covering space and a canonical projection (inducing discrete fibers).
We here explicitly consider the canonical projections $\pi_x$ which are for $x \in \mathcal X$ given by
\[
   \pi_x(y) = \pi_x(y_\mathbb S,y_\mathbb R) = (\exp_{x_\mathbb S}(y_\mathbb S),y_\mathbb R),\quad y\in\mathcal Y,
\]
i.e.\ the linear space components remain unchanged and the cyclic components
undergo the exponential mapping component-wise meaning that
\[
\exp_{x_\mathbb S}(y_\mathbb S) = (\exp_{x_1}(y_1),\ldots,\exp_{x_m}(y_m)).
\]
It is well known that continuous mappings to the base space have a lifting to the covering space. The lifting is uniquely determined by specifying $\pi^{-1}(x)$ for only one $x$. This lifting construction also applies to discrete mappings
$g\colon\Omega_0 \to \mathcal X$  defined on the rectangular grid $\Omega_0$ whenever $d_\infty(g)<\pi.$ We record this observation for further use.

\begin{lemma}\label{lem:LiftingOnWholeDomain}
	Let $g\colon\Omega_0 \to \mathcal X$ be an image with $d_\infty(g) < \pi$,
	and consider $q \in \mathcal X$ fulfilling~%
	\(d_{\mathbb S^1}(q_i, (g_{1,1})_i)<\pi\), \(i=1,\ldots,m\), i.e.\ no pair of
	cyclic data components is antipodal.
	We choose $\tilde g_{1,1} \in \mathcal Y$ such
	that~\(\pi_q(\tilde g_{1,1} ) = g_{1,1}\).
	Then there exists a unique lifted image \(\tilde g\colon\Omega_0 \to \mathcal Y \) such that 
	$\pi_q \tilde g = g$ holds component-wise and $d_\infty(g)<\pi.$
\end{lemma}

Next, we lift the inpainting functionals and derive relations between the
lifted and not lifted functionals and lifted and non lifted discrete functions.
To precisely formulate these relations we need some preparation.

For \( \delta>0\) and data $f\colon\Omega^C \to \mathcal X$ given on the complement of
the inpainting region $\Omega$, we consider the class
${\mathcal S}^\Omega(f,\delta)$ of grid functions $x$ defined on the whole domain $\Omega_0,$
which we define by
\begin{align}
{\mathcal S}^\Omega(f,\delta)
&=
\bigl\{ 
x\colon\Omega_0 \to \mathcal X
:\,
e_{\infty}(x,f) \leq \delta
\bigr\},\nonumber\\
\intertext{where}
\label{eq:def_e}
e_{\infty}(x,f) &\coloneqq \max_{(i,j) \in \Omega_0} d_{\mathcal X,m}(x_{i,j},f_{\nu(i,j)}),
\end{align}
and the mapping
\begin{equation}
\begin{split}
	\label{eq:DefNearestNeigh}
	\nu\colon\Omega_0 \to \Omega^C \text{ assigns }& (i,j) \in \Omega_0\\
	& \text{ a nearest neighbor in } \Omega^C.
\end{split}
\end{equation}
Here we measure the vicinity with respect to the distance on the grid induced
by taking shortest paths with respect to the eight-neighborhood.
The $x$ specified this way are `near' to the images $f$ on \(\Omega^C\) and
do not vary too much in \(\Omega\).
We also need an extension operator $E$ extending a function \(f\) defined on
$\Omega^C$ to a function \(E(f)\) defined on $\Omega_0$.
A particularly simple extension operator is the nearest neighbor operator $E_\nu$, defined, for all $(i,j),$ by
\begin{equation}\label{eq:DefEn}
E_\nu f_{i,j} = f_{\nu(i,j)}  
\end{equation}
with $\nu$ as in~\eqref{eq:DefNearestNeigh}.
Then there is a constant $B_\nu$, independent of $f$ but dependent on $\Omega$,
such that 
\begin{equation}\label{eq:ExtensionEst}
d_\infty(E_\nu(f)) \leq B_\nu(\Omega) d^\Omega_\infty(f).
\end{equation}
We now consider the `lifted' inpainting functionals.
We first note that we may write the problem~\eqref{eq:Functional4InpaintingNoiselessCombCyclicVec}
in the form~\eqref{eq:Functional4InpaintingNoisyCombCyclicVec} modifying the distance term to be infinite if
$x \neq f$ on $\Omega^C.$ Then, on the universal covering space $\mathcal Y$ of $\mathcal X,$ 
the inpainting functional $\tilde J_\Omega$ reads
\begin{equation}\label{eq:Functional4InpaintingCovering}
\begin{split}
\tilde J_{\Omega}(x) &=
\widetilde{F}_{\Omega^C}(x; \tilde f)
+ \alpha \widetilde{\operatorname{TV}}_{1}(x)
\\&\qquad + \beta \widetilde{\operatorname{TV}}_{2}(x)
+ \gamma \widetilde{\operatorname{TV}}_{1,1}(x),
\end{split}
\end{equation}
where \(\tilde f\colon\Omega^C \to\mathcal Y\) is a lifted image of~\(f\).
The functionals $\widetilde{F},$  $\widetilde{\operatorname{TV}}_1,$
$\widetilde{\operatorname{TV}}_{2}$ $\widetilde{\operatorname{TV}}_{1,1}$
are given by the corresponding functionals of Section~\ref{subsec:ModelVectorSpace} (there denoted without tilde) using $\mathcal Y$ as the underlying vector space.
We get the following relations:
\begin{lemma}\label{lem:ConvLem2}
	Let~$f\colon\Omega^C \to \mathcal X$
	with~\( d^\Omega_\infty(f) < \frac{\pi}{8 B_\nu(\Omega)}\) be given and let $q \in \mathcal X $ be a point not antipodal to $f_{\nu(1,1)}$ in any cyclic component. Choose a point $\tilde f_{\nu(1,1)}$ with $\pi_q(\tilde f_{\nu(1,1)}) = f_{\nu(1,1)}$ and let \( \tilde f \) denote the lifting of
	\( E_\nu f \).
		
	Then every  $x \in {\mathcal S}^\Omega(f,\frac{\pi}{8})$ has a unique
	lifting~$\tilde{x}$ w.r.t.\ the base point $q$ fulfilling
	$|\tilde x_{\nu(1,1)} - \tilde f_{\nu(1,1)}|_{\mathcal Y,m} \le \frac{\pi}{8}$.
	Furthermore,
		\begin{equation} \label{eq:LiftJ}
		J_\Omega(x)= \tilde{J}_\Omega(\tilde{x}) \quad \mbox{for all} \quad x \in {\mathcal S}^\Omega(f,\frac{\pi}{8}),
		\end{equation}
	where $J_\Omega$ either denotes the functional from~\eqref{eq:Functional4InpaintingNoiselessCombCyclicVec}
	or from~\eqref{eq:Functional4InpaintingNoisyCombCyclicVec}
	and $\tilde J_\Omega$ is its analogue in $\mathcal Y$ by~\eqref{eq:Functional4InpaintingCovering}.
\end{lemma}
\begin{proof}
	Let us consider $x \in {\mathcal S}^\Omega(f,\frac{\pi}{8}).$
	For $(k,l)$ in the eight-neighborhood of $(i,j),$ we have
	\begin{align}
			d_{\mathcal X,m}(x_{i,j}, x_{k,l}) 
		& \leq d_{\mathcal X,m}(x_{i,j}, f_{\nu(i,j)})
		\\&\qquad
			+ d_{\mathcal X,m}(f_{\nu(i,j)}, f_{\nu(k,l)})
		\\&\qquad
			+ d_{\mathcal X,m}(f_{\nu(k,l)}, x_{k,l})\nonumber\\
		& \leq \frac{2\pi}{8} + B_\nu(\Omega) \ d^\Omega_\infty(f)
			< \frac{3\pi}{8}.\label{eq:EstX}
	\end{align}
	By assumption, we
	have~$d_{\mathcal X,m}(x_{\nu(1,1)},f_{\nu(1,1)}) \leq \frac{\pi}{8}$.
	\\Therefore, every $x \in {\mathcal S}(f,\tfrac{\pi}{8})$ has a unique
	lifting~$\tilde{x}$ by Lemma~\ref{lem:LiftingOnWholeDomain} w.r.t.\ to the
	base point~$q$ fulfilling~$\norm{\tilde x_{1,1} - \tilde f_{1,1}}
	\leq \frac{\pi}{8}$.
	
	In order to show~\eqref{eq:LiftJ} we show equality for each of the
	involved summands. First, we consider $\operatorname{TV}_1$.
	By Lemma~\ref{lem:LiftingOnWholeDomain} we have 
	\(d_{\mathcal X,m}(x_{i,j},x_k)%
		= \lvert \tilde x_{i,j} - \tilde x_k\rvert_{\mathcal Y,m}%
		,\ k\in\{(i,j+1),(i+1,j),(i+1,j+1)\}\).
	Hence the definitions of $\operatorname{TV}_1$
	and~$\operatorname{\widetilde{TV}}_1$ imply
	$\operatorname{TV}_1(x) = \operatorname{\widetilde{TV}}_1(\tilde{x})$.
	Concerning second order differences, we first consider  
	the expressions~$D_2(x_{i-1,j},x_{i,j},x_{i+1,j}).$
	Similar to~\eqref{eq:EstX}, we have
	\(d_{X,m}(x_{i-1,j},x_{i+1,j}) < \frac{\pi}{2} \)
	which implies that the distance between any two members of the triple is
	smaller than $\frac{\pi}{2}$.	
	Due to the properties of the lifting~\(\tilde x\) this implies
	\(D(\tilde x_{i-1,i},\tilde x_{i,j}, \tilde x_{i+1,j}; b_2) < \pi\).
	The same argument applies to $D_2(x_{i,j-1},x_{i,j},x_{i,j+1})$
	which yields the equality for the $\operatorname{TV}_2$ terms.
	A similar argument shows that \( \operatorname{TV}_{1,1}(x)
	= \operatorname{\widetilde{TV}}_{1,1}(\tilde{x})\).
	Concerning the data term \(F(x;f)\) we define
	\(r_{i,j} = d_{\mathcal X,m}(x_{i,j},f_{\nu(i,j)})\)
	and 
	\(\tilde r_{i,j} =%
	\lvert \tilde x_{i,j} - \tilde f_{\nu(i,j)}\rvert_{\mathcal Y,m}\).
	We show that $r_{i,j}$ and $\tilde r_{i,j}$ agree for any $(i,j) \in \Omega^C$.
	By definition of ${\mathcal S}^\Omega(f,\delta)$, we have
	\(r_{i,j} \leq \frac{\pi}{8}\) for all $(i,j) \in \Omega_0.$
	Furthermore, by the construction of
	\(\tilde f\) and \(\tilde x\) it holds \(\tilde r_{i,j} = r_{i,j} + 2\pi k_{i,j}\),
	with \(k_{i,j}\in\mathbb N\) and \(k_{\nu(1,1)} = 0\). 
	We estimate
	\( \abs{\tilde r_{i,j+1}-\tilde r_{i,j}} =
	\abs[\big]{%
		\lvert \tilde x_{i,j+1} - \tilde f_{\nu(i,j+1)}\rvert_{\mathcal Y,m}
	-
		\lvert\tilde x_{i,j} - \tilde f_{\nu(i,j)}\rvert_{\mathcal Y,m}
	} \leq \frac{\pi}{4}\).
	If \(k_{i,j}\neq k_{i,j+1}\), then there exists
	\(k\in\mathbb Z\backslash\{0\}\)
	such that
	\[
		\abs{\tilde r_{i,j+1}-\tilde r_{i,j}}
		= \abs{\tilde r_{i,j+1}-\tilde r_{i,j} + 2\pi k}
		\geq 2\pi- \tfrac{\pi}{4} > \tfrac{\pi}{4}.
	\]
	This is a contradiction and therefore $k_{i,j}= k_{i,j+1}$.
	Similarly we conclude
	$k_{i,j}= k_{i+1,j}$.
	Hence,
	\(k_{i,j} = k_{\nu(1,1)} = 0\) for all \((i,j)\in \Omega_0\)
	which implies
	\(r_{i,j}=\tilde r_{i,j}\) for all $(i,j) \in \Omega^C$
	and completes the proof.
\end{proof}

To formulate the next lemma we need the quantity $d^\Omega_1$
for functions defined on the complement of the inpainting region.
We consider $f\colon\Omega^C \to \mathcal X,$ and define $d^\Omega_1(f)$
in analogy to~\eqref{eq:defSupNormOfDiff} by 
\begin{equation}\label{eq:defD1}
d^\Omega_1(f) = 
	\sum_{(i,j) \in \Omega^C}
	\max_{(k,l) \in \mathcal N_{i,j} \cap \Omega^C}
	d_{\mathcal X,m}(f_{i,j},f_{k,l}),
\end{equation}
with $\mathcal N_{i,j}$ being again the eight-neighborhood of \((i,j)\).
For the nearest neighbour extension operator $E_\nu$ defined in~\eqref{eq:DefEn}
we have the following estimate: there is a constant $C_\nu$, independent of $f$
but dependent on $\Omega$ (and on $\Omega_0$), such that
\begin{equation}\label{eq:ExtensionEst2}
d_1(E_\nu(f)) \coloneqq d^\emptyset_1(E_\nu (f)) \leq C_\nu(\Omega) d^\Omega_1(f).
\end{equation}

	As a further preparation, we need the following observation.  
	We consider the pure inpainting functional \eqref{eq:Functional4InpaintingNoiselessCombCyclicVec}.
	We let $u^*$ be a minimizer of \eqref{eq:Functional4InpaintingNoiselessCombCyclicVec} for given data $f.$ 
	Then there is 	
	a constant $B'_\nu(\Omega) \geq 1$ which depends on $\Omega,$ but not of $f,$ such that  
	\begin{equation}\label{eq:ds}
	d_{\mathcal X,m}(u^\ast_{i,j},f_{i,j})  \leq B'_\nu(\Omega) \ d^\Omega_\infty(f).
	\end{equation}
	For functions $f$  with small values $d^\Omega_\infty(f),$ this estimate follows from the boundedness of second differences 
	by first differences. For the remaining $f,$ the estimate \eqref{eq:ds} follows from the boundedness
	of the sphere $\mathbb S^1$ as a set. For vector space valued data, the boundedness of second differences 
	by first differences implies the estimate  for arbitrary input.	
	We note that we intentionally choose the symbol $u^*$ to avoid confusion when applying \eqref{eq:ds}.
	
\begin{lemma}\label{lem:ConvLem1}
	Let $\varepsilon>0$ and consider $f\colon\Omega^C \to \mathcal X$ with 
	$d^\Omega_\infty(f) < \tfrac{\varepsilon}{4 B'_\nu(\Omega)} $ .	
	We define
	\(p \coloneqq \max\{\alpha_1,\ldots,\alpha_4,\beta_1,\beta_2,\gamma\} \).
	 and assume that $p$ is so small that 
	\begin{align}
		\label{eq:EstAlpaBeta}
		d^\Omega_1(f) \leq \frac{\varepsilon^2}{8^2 \cdot 20 p C_\nu(\Omega) B'_\nu(\Omega)^2}	
	\end{align}
	where $C_\nu(\Omega)$ is given by~\eqref{eq:ExtensionEst2}.
	Then any minimizer~$x^\ast$ of the inpainting functional $J_\Omega$ given
	in~\eqref{eq:Functional4InpaintingNoiselessCombCyclicVec}
	or~\eqref{eq:Functional4InpaintingNoisyCombCyclicVec} fulfills
	\begin{equation}
		\label{eq:ResultLemma1}
		e_{\infty}(x^\ast,f)\leq \varepsilon.
	\end{equation}
\end{lemma}

%
\begin{proof}
	For a minimizer \(x^*\) of $J_\Omega$ given
	by~\eqref{eq:Functional4InpaintingNoiselessCombCyclicVec}
	or by~\eqref{eq:Functional4InpaintingNoisyCombCyclicVec}, 
	we consider $x'$ which we define as the closest point extension $E_\nu$ of the restriction $x^\ast|_{\Omega^C}$ 
	to the non-inpainting region $\Omega^C.$
	In view of  the definition of $e_\infty$ in~\eqref{eq:def_e} we consider the estimate
	\begin{equation}\label{eq:triangle4dec}
		\begin{split}
		d_{\mathcal X,m}(x^\ast_{i,j},f_{\nu(i,j)})
			&\leq d_{\mathcal X,m}(x^\ast_{i,j},x'_{i,j})\\
			&\qquad + d_{\mathcal X,m}(x'_{i,j},f_{\nu(i,j)}).
		\end{split}
	\end{equation}
	We start with the second summand on the right hand side and notice that 
	$x'_{i,j} = x^\ast_{i,j}$ for $(i,j) \in \Omega^C$
	which implies $d_{\mathcal X,m}(x'_{i,j},f_{i,j}) = d_{\mathcal X,m}(x^*_{i,j},f_{i,j})$ 
	for all $(i,j) \in \Omega^C.$
	We extend the data $f$ given on $\Omega^C$ to a grid function $g$ defined on
	$\Omega_0$ by setting~$g_{i,j} = (E_\nu f)_{i,j} = f_{\nu(i,j)}$ for all $(i,j)\in\Omega$. We get the estimate	
	\begin{align}\label{eq:EstJOfMin}
	J_\Omega(x^*) \leq  J_\Omega(g)
	& = \alpha \operatorname{TV}_{1}(g)
	+ \beta \operatorname{TV}_{2}(g)
	+ \gamma \operatorname{TV}_{1,1}(g).
	\end{align}
	We further estimate the right hand side of~\eqref{eq:EstJOfMin}: since
	the second order differences may be estimated by two times the first order
	differences, we get
	$\beta \operatorname{TV}_{2}(g) \leq 2 (\max_i \beta_i)(1,1,0,0) \operatorname{TV}_{1}(g),$
	and an analogous inequality for $\gamma \operatorname{TV}_{1,1}(g)$.
	Hence
	\begin{align}\label{eq:EstJOfg}
	 J_\Omega(g)
		\leq 5 \max\{\alpha_1,\ldots,\alpha_4,\beta_1,\beta_2,\gamma\} \ (1,1,1,1) \operatorname{TV}_{1}(g).
	\end{align}
	Next, we estimate each summand appearing in $\operatorname{TV}_{1}(g)$ by
	the corresponding summand in $d_1(g)$ to conclude
	that~$(1,1,1,1) \operatorname{TV}_{1}(g) \leq 4 d_1(g)$.
	We use~\eqref{eq:ExtensionEst} to get 
	\begin{align}\label{eq:EstJEnd}
	J_\Omega(g)	\leq 20 p d_1(g)  \leq 20 p C_\nu(\Omega) d^\Omega_1(f).
	\end{align}
	As a consequence we obtain for~$(i,j) \in \Omega^C$
	\begin{equation}
		\label{eq:EstST}
		\begin{split}
		d_{\mathcal X,m}(x'_{i,j},f_{i,j})^2
		&=	d_{\mathcal X,m}(x_{i,j}^*,f_{i,j})^2 
		\leq J_\Omega(x^*)  \\
		&	\leq  J_\Omega(g)
			\leq 20 p C_\nu(\Omega) d^\Omega_1(f)\\
		&	\leq (\tfrac{\varepsilon}{8 B'_\nu(\Omega) })^2 
			\leq (\tfrac{\varepsilon}{8})^2.
		\end{split}
	\end{equation}
	By the definition of $x',$
	this implies that for all $(i,j) \in \Omega_0,$ we have
	\begin{equation}\label{eq:ds1}
	d_{\mathcal X,m}(x'_{i,j},f_{\nu(i,j)})	\leq \tfrac{\varepsilon}{8}.
	\end{equation}
	Next, we look at the first summand in~\eqref{eq:triangle4dec}. 	
	Since $x'_{i,j} = x^\ast_{i,j}$ for $(i,j) \in \Omega^C,$
    we may restrict to estimate 
    $d_{\mathcal X,m}(x'_{i,j},x^\ast_{i,j})$ on $\Omega.$
    For this purpose, we consider the pure inpainting problem \eqref{eq:Functional4InpaintingNoiselessCombCyclicVec}
    for the inpainting region $\Omega$ for data $x^\ast$.	
    Using \eqref{eq:ds} with data $f=x^\ast$ we have that 
	\begin{equation}\label{eq:ds2}
		\begin{split}
			&d_{\mathcal X,m}(x'_{i,j},x^\ast_{i,j})\\
			&\ \ \leq B'_\nu(\Omega) \ d^\Omega_\infty(x')\\
         &\ \ \leq B'_\nu(\Omega) (d^\Omega_\infty(f) + 
         2 \max_{(i,j) \in \Omega} d_{\mathcal X,m}(x'_{i,j},f_{\nu(i,j)}))
			\\
			&\ \ \leq \tfrac{\varepsilon}{2}.
		\end{split}
  	\end{equation}
   	For the second inequality we used \eqref{eq:EstST}.	
	Combining the estimates~\eqref{eq:ds1} and~\eqref{eq:ds2}, we get
	\begin{equation*}
		\begin{split}
			d_{\mathcal X,m}(x^\ast_{i,j},f_{\nu(i,j)})
		&\leq d_{\mathcal X,m}(x^\ast_{i,j},x'_{i,j})
		+ d_{\mathcal X,m}(x'_{i,j},f_{\nu(i,j)})
		\\
		&\leq \tfrac{\varepsilon}{2}+\tfrac{\varepsilon}{2}
		= \varepsilon
	\end{split}
	\end{equation*}
	which implies that~\[e_{\infty}(x^\ast,f) = \max_{(i,j) \in \Omega_0}
	\{d_{\mathcal X,m}(x^\ast_{i,j},f_{\nu(i,j)})\} \leq \varepsilon\]
	and this finishes the proof. 
\end{proof}
\begin{lemma}\label{lem:TildeJ}
	The statement from Lemma~\ref{lem:ConvLem1} does also hold for data $f\colon\Omega^C \to \mathcal Y$
	and the inpainting functionals $\tilde J_\Omega$ given by~\eqref{eq:Functional4InpaintingCovering}.
\end{lemma}
\begin{proof}
	The statement is obtained following the lines of the proof of Lemma~\ref{lem:ConvLem1}.
\end{proof}
Now we combine Lemma~\ref{lem:ConvLem2} and~\ref{lem:ConvLem1} to locate
the minimizers of $J$ and $\tilde J$. 
This part of the proof is rather similar to \cite{bergmann2014second}
which is the reason for streamlining it.

\begin{lemma} \label{lem:ConvLem3}
	Consider $\varepsilon$ with $0<\varepsilon<\frac{\pi}{8}$ and 
	let the function $f\colon\Omega^C \to \mathcal X$
	fulfill both $d^\Omega_\infty(f) <\frac{\pi}{8 B_\nu(\Omega)}$
	as well as 	$d^\Omega_\infty(f) < \tfrac{\varepsilon}{4 B'_\nu(\Omega)}$ 	
	and assume that the parameters $\alpha,\beta,\gamma$ of $J_\Omega$
	from~\eqref{eq:Functional4InpaintingNoiselessCombCyclicVec}
	or~\eqref{eq:Functional4InpaintingNoisyCombCyclicVec}  
	fulfill ~\eqref{eq:EstAlpaBeta} w.r.t.\ $\varepsilon$.

	Then any minimizer $x^\ast$ of $J_\Omega$ lies
	in~\({\mathcal S}^\Omega(f,\tfrac{\pi}{8}) \).
	Furthermore, denote by $\tilde{f}$ the unique lifting of $f$ w.r.t.\ a base
	point~$q$ and fixed~$\tilde{f}_{\nu(1,1)}$
	with~$\pi_q (\tilde{f}_{\nu(1,1)}) = f_{\nu(1,1)}$.
	Then each minimizer~$y^\ast$ of~$\tilde{J}_\Omega$ defines a
	minimizer $x^\ast\coloneqq\pi_q(y^\ast)$ of~$J_\Omega$.
	Conversely, the uniquely defined lifting $\tilde{x}^\ast$ of a
	minimizer~$x^\ast$ of $J_\Omega$ is a minimizer of $\tilde{J}_\Omega$.
\end{lemma}
\begin{proof}
	By Lemma~\ref{lem:ConvLem1}, any minimizer~$x^\ast$ of the inpainting
	functional $J_\Omega$ fulfills $e_{\infty}(x^\ast,f) \leq \varepsilon.$
	Since  $\varepsilon< \tfrac{\pi}{8},$ we
	get~$x^\ast \in {\mathcal S}^\Omega(f,\tfrac{\pi}{8})$.
	For the second statement we notice that the
	mapping~$x \mapsto \tilde{x}$ is a bijection
	from~\({\mathcal S}^\Omega(f,\delta) \) to the
	set~\({\widetilde {\mathcal S}}^\Omega(\tilde f,\delta) 
	\coloneqq \{y\colon\Omega_0 \to \mathcal Y \colon|y_{i,j}-\tilde f_{\nu(i,j)}|_{\mathcal Y,m}<\delta \}.
	\)
	for every $\delta$ with $0<\delta\leq \tfrac{\pi}{8}$.
	In particular, we may choose $\delta = \tfrac{\pi}{8}.$
	If $y^\ast\in\mathcal Y$ is a minimizer of $\tilde{J}_\Omega$, it lies
	in~${\widetilde {\mathcal S}}^\Omega(\tilde f,\tfrac{\pi}{8})$ by
	Lemma~\ref{lem:TildeJ}. By~\eqref{eq:LiftJ} and the minimizing property
	of~$y^\ast$ we obtain for any~$x \in {\mathcal S}^\Omega(f,\tfrac{\pi}{8})$ that
	\(
	J_\Omega(\pi_q(y^\ast)) = \tilde{J}_\Omega(y^\ast) \leq \tilde{J}_\Omega(\tilde{x}) = J_\Omega(x).
	\)
	As a consequence, $\pi_q(y^\ast)$ is a minimizer of~$J_\Omega$
	on~${\mathcal S}^\Omega(f,\tfrac{\pi}{8})$.
	By Lemma~\ref{lem:ConvLem1} all the minimizers of $J_\Omega$ are contained
	in~${\mathcal S}^\Omega(f,\tfrac{\pi}{8})$ and hence $\pi_q(y^\ast)$ is a minimizer
	of~$J_\Omega$
	.
	For the last statement let~$x^\ast$ be a minimizer of~$J_\Omega$.
	For its lifting~$\tilde{x}^\ast$ and
	any~$\tilde y \in \widetilde{S}^\Omega(\tilde f,\tfrac{\pi}{8})$, we get
	\(
		\tilde{J}_\Omega(\tilde{x}^\ast) = J_\Omega (x^\ast)
		\leq  J_\Omega(\pi_q(\tilde y)) = \tilde{J}_\Omega(\tilde y).
	\)
	Thus,~$\tilde{x}^\ast$ is a minimizer of~$\tilde{J}_\Omega$
	on~${\widetilde {\mathcal S}}^\Omega(\tilde f,\tfrac{\pi}{8})$.
	Since by Lemma~\ref{lem:TildeJ} all minimizers of $\tilde{J}$ lie
	in~${\widetilde {\mathcal S}}^\Omega(\tilde f,\tfrac{\pi}{8})$, the last assertion follows.
\end{proof}

After establishing relations between these functionals and their lifted versions,
grid functions and data, we next formulate a convergence result for vector space data in $\mathcal Y.$
It is a reformulation of a convergence result which can be found in~\cite{Bac13a} for the more general class of Hadamard spaces or which can be derived from~\cite{Ber10}.
\begin{theorem} \label{thm:bacak}
Let $J = \sum_{l=1} ^c J_l$, with each $J_l$ being a proper, closed, convex
functional on $\mathcal Y^{\Omega_0}$ and assume that $J$ has a global
minimizer. Assume further that there is $L >0$ such that the iterates $\{x^{(k+\frac{l}{c})} \}$ of the CPPA, cf. Algorithm~\ref{alg:CPPA}, fulfill
\[
J_l(x^{(k)}) - J_l (x^{(k+\frac{l}{c})}) \le L \norm[\big]{x^{(k)} - x^{(k+\frac{l}{c})}}, \quad l=1,\ldots,c,
\]
for all $k\in \mathbb N_0$.
Then the sequence $\{x^{(k)} \}_k$ converges to a minimizer of $J$.
In particular, the iterates fulfill
\begin{align} \label{noch_2}
\norm[\big]{x^{(k+\frac{l-1}{c})} - x^{(k+\frac{l}{c})}} \leq 2 \lambda_k L,
\end{align}
and, for all $x \in \mathcal Y^{\lvert \Omega_0\rvert },$ 
\begin{equation}
		\begin{split}
	\norm[\big]{x^{(k+1)}-x}^2
	&\leq
	\norm[\big]{x^{(k)}-x}^2
	\\&\qquad - 2 \lambda_k
		\bigl(J(x^{(k)})-J(x)\bigr)
	\\&\qquad
		+ 2 \lambda_k^2 L^2 c(c+1) \label{noch_1}.
		\end{split} 
	\end{equation}
\end{theorem}
Next we locate the iterates of the CPPA for vector space data in $\mathcal Y$
on a ball whose radius can be controlled. 
Since the data $f\colon\Omega^C \to  \mathcal Y$ is not defined on the whole grid
$\Omega_0$, we incorporate an extension operator $E$, e.g.\ \(E_\nu\).
An extension is needed as an initialization of the CPPA.
We note that there is a positive number $L'$ such that the
iterates~$\bigl\{ x^{(k+\frac{l}{c})} \bigr\}$ produced by
Algorithm~\ref{alg:CPPA} fulfill
\begin{equation}\label{eq:boundOnDistData}
	\norm[\big]{E(f) - x^{(k+\frac{l}{c})}}[\infty] \leq L'.
\end{equation}
This can be seen by taking $m_0,M_0\in\mathbb R$ as the minimum and maximum of all components and pixels of~$E(f)$, respectively,
	and by letting $m_{k+\frac{l}{c}},M_{k+\frac{l}{c}}\in\mathbb R$ the corresponding minima and maxima of the iterates  $x^{(k+\frac{l}{c})}.$
	Looking at the concrete form of the proximal mappings, the only proximal mappings which can increase the maximal value
	or decrease the minimal value during the iteration are those of the second order differences. 
	The possible increase is immediately decreased to initial niveau (of the macro-step) by the proximal mappings of data and first order terms if outside $[C \cdot m_0,C \cdot M_0]$ for sufficiently large $C.$  
	We note that $L'$ does not depend on the particular vector-valued $f$ but only on $M_0-m_0.$

\begin{lemma}\label{lem:ConvLem4}
Let~$f\colon\Omega^C \to \mathcal Y$ and a parameter
sequence $\lambda = \{\lambda_k\}_k$ of the CPPA with
property~\eqref{eq:CPPAlambda} be given.
Further let $\bigl\{ x^{(k+\frac{l}{c})} \bigr\}$ be the sequence
produced by Algorithm~\ref{alg:CPPA} for the inpainting functionals
$\tilde{J}_\Omega$ given by~\eqref{eq:Functional4InpaintingCovering}.
Let $x^\ast\colon\Omega_0 \to \mathcal Y$ be a minimizer of
$\tilde{J}_\Omega$.
Then, for all $k \in \mathbb N _0$ and all $l \in \{1,\ldots,c\}$, we have
\begin{align}
\label{eq:estCPPAiterates}
	&\norm[\big]{x^{(k+\frac{l}{c})}-x^\ast}
	\leq R
\intertext{where}
&R
 \coloneqq
	\sqrt{ \norm{E(f)-x^\ast}^2 + 2 \norm{\lambda}^2 L^2 c (c+1)} + 2 \norm{\lambda}[\infty] cL,
\end{align}
where $L= \max(4,L')$ using \(L'\) from~\eqref{eq:boundOnDistData} and $c$ denotes the number of inner iterations, i.e.
$c= 18$ in case of~\eqref{eq:Functional4InpaintingNoiselessCombCyclicVec} and
$c= 19$ in case of~\eqref{eq:Functional4InpaintingNoisyCombCyclicVec}, respectively, and $E$ is an operator extending $\mathcal Y$ valued functions defined on $\Omega^C$ to $\Omega_0$
used for initializing the algorithm.
Here, $\norm{\lambda}^2 = \sum_i \lambda_i^2$ and $\norm{\lambda}[\infty] = \sup_i \lambda_i.$  
\end{lemma}
\begin{proof}
Equation~\eqref{noch_1} in Theorem~\ref{thm:bacak} tells us that
\begin{equation}\label{eq:appliedThm41}
	\begin{split}
		\norm[\big]{x^{(k+1)}-x}^2
		&\leq
			\norm[\big]{x^{(k)}-x}^2\\
		&\qquad - 2 \lambda_k \bigl(\tilde{J}(x^{(k)}\bigr)-\tilde{J}(x)]\\
		&\qquad + 2 \lambda_k^2 L^2 c(c+1).
	\end{split}
\end{equation}
We choose $L''$ as the maximum of the Lipschitz constants 
for the terms $J_i$ of Section~\ref{subsec:InpaintDenoiseCPPA} originating from the splitting  of the terms \(\operatorname{\widetilde{TV}}_1\),
\(\operatorname{\widetilde{TV}}_2\)
and~\(\operatorname{\widetilde{TV}}_{1,1}\). 
We note that these Lipschitz constants which are formed w.r.t. the Euclidean norm on $\mathcal Y$ are all bounded by $4$ which, for the second order differences,  
is seen by estimating second order difference by sums of first order differences.
This implies $L'' \leq 4.$
For the quadratic data term, 
we may differentiate $x_{i,j} \mapsto \frac{1}{2}  \lvert f_{i,j} -x_{i,j}\rvert^2.$
Then we notice that the $x_{i,j}$ are confined to an $L'$ ball around $E(f)_{i,j}$ 
(cf. \eqref{eq:boundOnDistData})
which bounds the Lipschitz constant of the data term (where the metric on $\mathcal Y$ is the Euclidean metric) by $L'$ in this case.
Therefore, we can set $L = \max(L',4)$.
We apply~\eqref{eq:appliedThm41} with a minimizer $x = x^\ast$ to obtain the estimate
\begin{align}
	\norm[\big]{&x^{(k+1)}-x^\ast}^2\\
	&\quad\leq \norm[\big]{x^{(k)}-x^\ast}^2 + 2 \lambda_k^2 L^2 c(c+1)\\
	&\quad\leq  \norm[\big]{x^{(0)}-x^\ast}^2 + 2 \sum\nolimits_{j=0}^k \lambda_j^2 L^2 c(c+1).
\end{align}
Using $x^{(0)}= E(f)$ yields $\norm{x^{(k+1)}-x^\ast}^2 \leq  \norm{ E(f)-x^\ast}^2 + 2 \norm{\lambda} \rVert_2^2 L^2 c(c+1).$
Now we use~\eqref{noch_2} and the triangle inequality to estimate 
\begin{equation}
	\begin{split}
	\norm{&x^{(k+\frac{l}{c})}-x^\ast}\\
		&\quad\leq 2  \lambda_k c L+ \norm{x^{(k)}-x^\ast}\\
		&\quad\leq 2  \lambda_k c L+ \sqrt{\norm{E(f)-x^\ast}^2 + 2 \norm{\lambda}^2 L^2 c(c+1)}.
	\end{split}
\end{equation}
This completes the proof.

\end{proof}
The following lemma shows that lifting commutes with applying the proximal mappings for the previous assumptions.
%
\begin{lemma}\label{lem:ConvLem5}
Let  $f\colon\Omega^C \to \mathcal X$ 
with~\( d^\Omega_\infty(f) < \frac{\pi}{8 B_\nu(\Omega)}\) and its lifting
$\tilde{f}$ of $f$ w.r.t.\ a base point~$q$ as before.
For each summand $J_l$ in the splitting $J_\Omega = \sum_l J_l$ from Section~\ref{subsec:InpaintDenoiseCPPA} for both inpainting functionals $J_\Omega$ defined via~\eqref{eq:Functional4InpaintingNoiselessCombCyclicVec}
and~\eqref{eq:Functional4InpaintingNoisyCombCyclicVec}, their corresponding functionals $\tilde J_\Omega$ from~\eqref{eq:Functional4InpaintingCovering},
any $x \in \mathcal S^\Omega(f,\frac{\pi}{8})$, 
and its lifting~$\tilde x$ w.r.t.\ $q$, we have
\begin{equation} \label{eq:commuteProxProj}
  \prox_{\lambda J_l}(x) = \pi_q(\prox_{\lambda \tilde{J}_l}(\tilde{x})),
\end{equation}
for all $l \in \{1,\ldots,18\}$ in case of~\eqref{eq:Functional4InpaintingNoiselessCombCyclicVec},	 
and for all $l \in \{1,\ldots,19\}$ in case of~\eqref{eq:Functional4InpaintingNoisyCombCyclicVec}.
\end{lemma}
\begin{proof}
The functional $J_{19}$ from~\eqref{eq:splitting-dataterm2DX} appearing in the splitting of~\eqref{eq:Functional4InpaintingNoisyCombCyclicVec} is based on the distance to the data $f$ for indices $(i,j) \in \Omega^C$.
Since $x \in \mathcal S^\Omega(f,\frac{\pi}{8})$, it holds $d_{\mathcal X,m}(x_{i,j},f_{i,j}) \leq \frac{\pi}{8}$ for all $(i,j) \in \Omega^C$.
The components of the proximal mapping $\prox_{\lambda J_{19}}$ are given by
Proposition~\ref{prop:data-prox} from which we
conclude~\eqref{eq:commuteProxProj} for $l=19$.
The other proximal mappings of $J_1,\ldots,J_{18}$, are given via proximal
mappings of the first and second order cyclic differences from Theorem~\ref{thm:proxy_w}.
We first consider cyclic components of the first order differences, i.e. summands involving $D_1$.
By the triangle inequality we have
\begin{align*}
	d_{\mathcal X,m}&(x_{i,j},x_{i,j+1})&\\%
	&\leq d_{\mathcal X,m}(x_{i,j},f_{\nu(i,j)})
		+ d_{\mathcal X,m}(f_{\nu(i,j)},f_{\nu(i,j+1)})\\
	&\qquad + d_{\mathcal X,m}(x_{i,j+1},f_{\nu(i,j+1)})\\
	&\leq \frac{2\pi}{8} + B_\nu(\Omega) \ d^\Omega_\infty(f)
	< \frac{3\pi}{8}.
\end{align*}
Analogously we get $d_{\mathcal X,m}(x_{i,j},x_{i+1,j}) < \frac{3\pi}{8}$.
By the explicit form of the proximal mapping given in
Theorem~\ref{thm:proxy_w} we obtain~\eqref{eq:commuteProxProj}
for the $J_l$, $l=1,\ldots,8$, which involve first order differences.
Next we consider the second order differences $D_2$ with respect to the cyclic
components.
Let us exemplarily consider the vertical second order difference~$D_2(x_{i,j-1},x_{i,j},x_{i,j+1})$.
Analogously as above, we see that the inequalities
$d_{\mathcal X,m}(x_{i,j-1},x_{i,j}) < \frac{3\pi}{8}$,
$d_{\mathcal X,m}(x_{i,j},x_{i,j+1})<\frac{3\pi}{8}$
and~$d_{\mathcal X,m}(x_{i,j-1},x_{i,j})<\frac{\pi}{2}$ hold.
Hence all the cyclic parts of the contributing values of~\(x\) lie in a common ball of radius $\pi/2$.
Applying the proximal mapping in Theorem~\ref{thm:proxy_w} the cyclic parts of the resulting points lie in a common open ball of radius $\pi$.

An analogous statement holds true for the horizontal part.
Hence the proximal mappings of these second differences agree with
the cyclic version under identification via $\pi_q$. This
implies~\eqref{eq:commuteProxProj} for $J_9,\ldots, J_{14}.$
It remains to deal with the cyclic components of the mixed second order
differences $D_{1,1}(x_{i,j},x_{i+1,j},x_{i,j+1},$ $x_{i+1,j+1})$.
As above, we have for neighboring data items that the distance on the cyclic parts is smaller
than~$\frac{3\pi}{8}$.
For all four contributing values of \(x\) we have that the
pairwise distance on the cyclic parts is smaller than $\frac{\pi}{2}$. 
So their cyclic parts again lie in a boll of radius smaller than \(\pi\)
and the proximal mappings agree under identification.
This completes the proof.
\end{proof}

In the following main theorem we combine the preceding
lemmas to show that the output of the applied proximal mappings remains in $\mathcal S^\Omega(f,\frac{\pi}{8})$. 
This then allows for an iterated application of Lemma~\ref{lem:ConvLem5}. 

%
\begin{theorem}\label{thm:Convergence}
We choose the sequence $\lambda = \{\lambda_k\}_k$ fulfilling property~\eqref{eq:CPPAlambda}
and \(\varepsilon>0\) such that 
\begin{align} \label{eq:condTheorem}
	 \sqrt{ 4 \varepsilon^2 + 2 \norm{\lambda}^2 L^2 c (c+1)}
	+ 2 \norm{\lambda}[\infty] cL  < \frac{\pi}{16},
\end{align}
where $c = 18$ or $c=19$ and $L= \max(4,L')$ with $L'$ as in~\eqref{eq:boundOnDistData}.
We consider data  $f\colon\Omega^C \to \mathcal X$ 
with both~\( d^\Omega_\infty(f) < \frac{\pi}{8 B_\nu(\Omega)}\)
and $d^\Omega_\infty(f) < \tfrac{\varepsilon}{4 B'_\nu(\Omega)}.$
We assume further  that the parameter vectors $\alpha,\beta,\gamma$
of the inpainting functionals  $J_\Omega$
given by either~\eqref{eq:Functional4InpaintingNoiselessCombCyclicVec}
or~\eqref{eq:Functional4InpaintingNoisyCombCyclicVec}
satisfy~\eqref{eq:EstAlpaBeta} and that the initialization of the inpainting region $E(f)$ is close to the nearest neighbor extension $E_\nu(f)$ in the sense that 
\[d_{\mathcal X,m}(E(f),E_\nu(f)) = \max_{i,j} d_{\mathcal X,m}(E(f)_{ij},E_\nu(f)_{ij}) \leq \varepsilon.
\]
Then the sequence $\{x^{(k)} \}_k$ generated by the
CPPA given by Algorithm~\ref{alg:CPPA} converges to a global minimizer of $J$. 
\end{theorem}
\begin{proof} 
Let $\tilde{f}$ be the lifting of $E_\nu(f)$ with respect to a base
point~$q$ not antipodal to $f_{\nu(1,1)}$ and fixed $\tilde{f}_{\nu(1,1)}$ with
$\pi_q (\tilde{f}_{\nu(1,1)}) = f_{\nu(1,1)}$.
Furthermore, let $\tilde J_\Omega$ denote the analogue of $J_\Omega$ for $\mathcal Y$ valued data given by~\eqref{eq:Functional4InpaintingCovering}.
Since $d_{\mathcal X,m}(E(f),E_\nu(f)) \leq \varepsilon \leq \frac{\pi}{32}$, the function $ E(f)$ is in
 ${\mathcal S}^\Omega(f,\frac{\pi}{32}) \subset$  ${\mathcal S}^\Omega(f,\frac{\pi}{8})$
 (which is important for the application of Lemma~\ref{lem:ConvLem2} to the grid function $E(f)$ later on.)
From Lemma~\ref{lem:ConvLem1} we conclude that the minimizer~$y^\ast$
of~$\tilde{J}_\Omega$ fulfills
$\norm{y^\ast - \tilde{f}} \leq \varepsilon < \frac{\pi}{32},$ where $\tilde{f}$ is defined at the beginning of this proof.
By~\eqref{eq:estCPPAiterates} we obtain
\begin{align} \label{eq:Nr1}
	R &= \sqrt{ \norm[\big]{y^\ast - \widetilde{E(f)}}^2
			+ 2 \norm{\lambda}^2 L^2 c (c+1)}
		+ 2 \norm{\lambda}[\infty] cL \\
	&\leq \sqrt{ 2 \norm{y^\ast - \tilde{f}}^2 +  2 \norm[\big]{\tilde{f} - \widetilde{E(f)}}^2 + 2 \norm{\lambda}^2 L^2 c (c+1)}
	\\&\qquad
	+ 2 \norm{\lambda}[\infty] cL\\
	&\leq \sqrt{ 4 \varepsilon^2 + 2 \norm{\lambda}^2 L^2 c (c+1)}
		+ 2 \norm{\lambda}[\infty] cL  < \frac{\pi}{16},
\end{align}
where \(\widetilde{E(f)}\) denotes the lifting of \(E(f)\).
By Lemma~\ref{lem:ConvLem4} the iterates $y^{(k+\frac{l}{c})}$ of 
the  CPPA for $\mathcal Y$-valued data fulfill
\[
  \norm{y^{(k+\frac{l}{c})} - y^\ast}  \le R < \frac{\pi}{16}.
\]
Hence, combining these estimates yields $\norm{y^{(k+\frac{l}{c})} - \tilde{f}} \leq \norm{y^{(k+\frac{l}{c})} - y^\ast}$ $+   \norm{y^\ast - \tilde{f}}$ $\leq \tfrac{3 \pi}{32}.$
Estimating the discrete $\ell^\infty$ norm by the $\ell^2$ norm, we get from the previous line that 
$\norm{y^{(k+\frac{l}{c})} - \tilde{f}}[\infty] < \tfrac{3 \pi}{32}$
which means that all iterates $y^{(k+\frac{l}{c})}$ stay within $\tilde{\mathcal S}^\Omega(\tilde f,\frac{\pi}{8}).$

After these preparations we now consider the
sequence~$\{ x^{(l+\frac{k}{c})} \}$ of the CPPA for~$\mathcal X$-valued data
$f$ with initialization $E(f)$.
We show that $x^{(k+\frac{l}{c})} = \pi_q(y^{(k+\frac{l}{c})})$.\\
By definition and by Lemma~\ref{lem:ConvLem3}, we know
\[
	x^{(0)} = E(f) = \pi_q \bigl(\widetilde{E(f)}\bigr) = \pi_q(y^{(0)}).
\]
We continue a proof by induction and assume that $x^{(k+\frac{l-1}{c})} = \exp_q(y^{(k+\frac{l-1}{c})})$.
By the local bijectivity of the lifting shown in Lemma~\ref{lem:ConvLem2},
and since $y^{(k+\frac{l-1}{c})}\in\tilde{\mathcal S}^\Omega(\tilde f,\tfrac{\pi}{8})$,
we conclude $x^{(k+\frac{l-1}{c})}\in \mathcal S^\Omega(f,\tfrac{\pi}{8})$.
\\
By Lemma~\ref{lem:ConvLem5}, we obtain
\begin{equation}
	\begin{split}
		\pi_q(y^{(k+\frac{l}{c})})
		&= \pi_q\bigl(\prox_{\lambda_k\tilde J_{l}}(y^{(k+\frac{l-1}{c})})\bigr)\\
		&= \prox_{\lambda_k J_{l}}( x^{(k+\frac{l-1}{c})}) = x^{(k+\frac{l}{c})}.
	\end{split}
\end{equation}
By the same argument as above we have again $x^{(k+\frac{l}{c})} \in \mathcal S(f,\tfrac{\pi}{8})$.
Finally, Theorem~\ref{thm:bacak} tells us that
\[
x^{(k)} = \pi_q(y^{(k)}) \rightarrow \pi_q(y^\ast) \quad {\rm as} \quad k \rightarrow \infty
\]
and by Lemma~\ref{lem:ConvLem3} we finally obtain that $x^\ast \coloneqq \pi_q(y^\ast)$ is a global minimizer of $J$.

\end{proof}
\section{Applications} %
\label{sec:Applications}
%
%
%
In this section we apply our algorithms to various image processing tasks. These are denoising in HSV space in Section~\ref{subsec:NLcolorspaceAppl},
inpainting in both a noise-free as well as a noisy setting in Section~\ref{subsec:Inpainting} for both synthetic as well as real world data. 
Finally, we apply our algorithms for denoising frames in volumetric phase-valued data ---in our case, frames of a 2D film--- in Subsection\ref{sec:denoisingVolumes}. Our approach is based on utilizing the neighbouring $l$ frames to incorporate the temporal neighbourhood. The idea generalizes to arbitrary data spaces and volumes consisting of layers of 2D data.

\begin{figure*}[t]\centering
		\includegraphics[scale=.9]{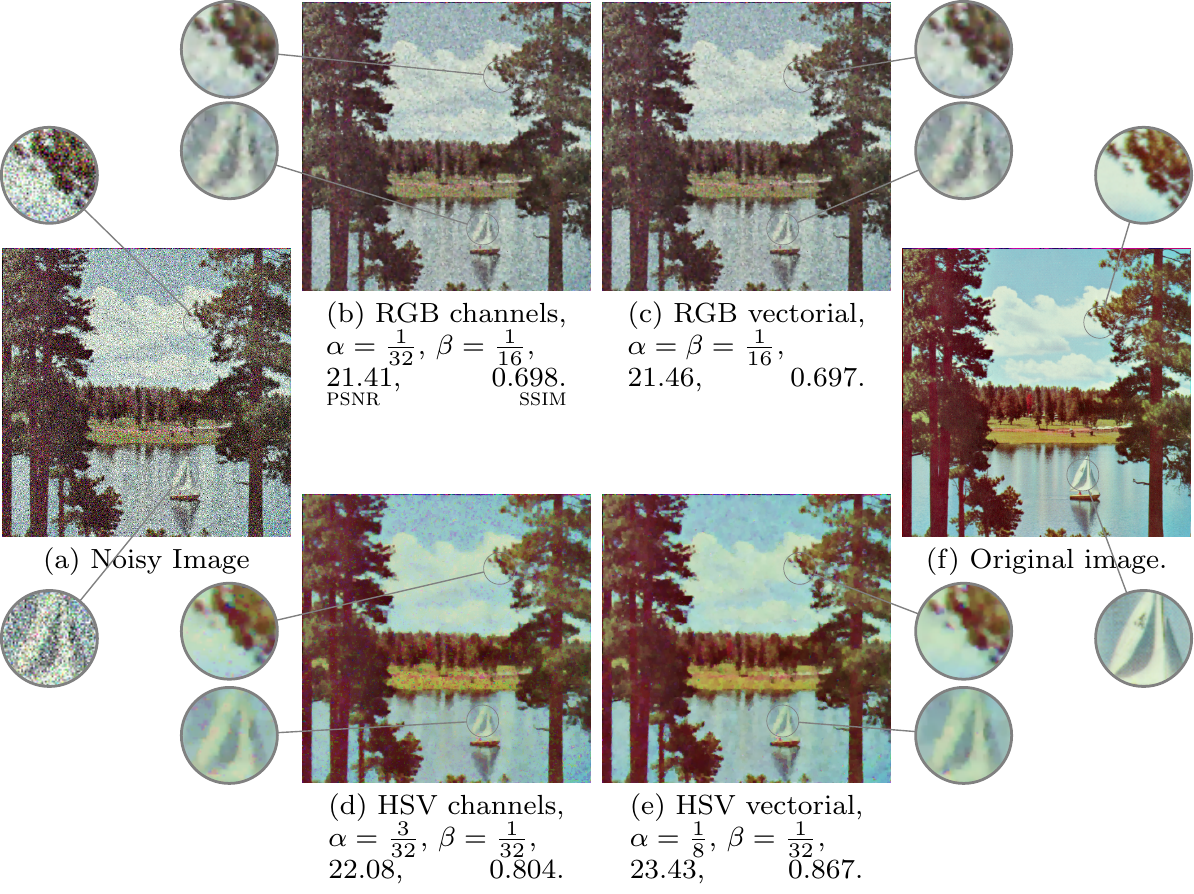}
	\caption{Denoising an image with independent wrapped Gaussian and Gaussian noise~\(\sigma=\tfrac{1}{5}\) on each of the HSV channels. 
	The RGB-based approaches (b), (c) produce less colorful results than the 
	HSV-based approaches (d),(e). In contrast to channel-wise denoising (d), the combined
	approach proposed in this paper (e) gets the object boundaries more properly and outperforms the other approaches in both PSNR ans SSIM.}\label{fig:sailboat}
\end{figure*}
The  algorithms were implemented in MATLAB 
The computations for the following examples were performed on a MacBook Pro with an Intel Core i5,
2.6\,Ghz and Mac OS 10.10.1. For all following experiments we set the sequence \(\{\lambda_k\}_k\) to \(\lambda_k = \frac{\pi}{2k}\) and set the number of iterations \(k\) to be \( 400 \) as a stopping criterion for Algorithm~\ref{alg:CPPA}.

\subsection{Denoising nonlinear color space data}
\label{subsec:NLcolorspaceAppl}
As a first application, we consider denoising color space data. Various nonlinear
color spaces have been considered in the literature; examples are 
luma plus chroma/chrominance based spaces such as YIQ, YUV and YDbDr
and HSL type color space such as HSL, HSI or HSD.
We here consider the HSV (hue-saturation-value) color space: 
the hue component is cyclic, the saturation and the value component are real-valued.

We apply our algorithm for denoising combined cyclic and linear data to these $\mathbb{S}^1 \times \mathbb R^2$ valued data. We compare the results with the usual approach using the linear RGB color space. For both spaces, we compare our approach on the product space with a model that denoises each channel separately. Finally, we compare the results of all these approaches under different noise models: we impose Gaussian noise on each component in RGB space. We impose Gaussian noise to saturation and value, and wrapped Gaussian noise, i.e. wrapping the normal distribution on the circle or in other words computing \(\bmod\,2\pi\) after adding the noise, to the hue component of the HSV space.

In Figure~\ref{fig:sailboat}, a colorful drawing of a
sailboat\footnote{Taken from the USC-SIPI Image Database, see
\url{http://sipi.usc.edu/database/database.php?volume=misc&image=14}} of
size~\(512\times512\) pixel is obstructed by noise on all three channels of the HSV
color space: for the hue, which is given on \([0,1]\) we applied wrapped
Gaussian noise (\(\bmod\,1\)) with standard deviation \(\sigma = \tfrac{1}{5}\) and mean \(0\), for saturation and
value, which are also given on the same range but are not cyclic, we applied Gaussian noise, also with \(\sigma = \tfrac{1}{5}\) and set all pixels exceeding \(1\) to \(1\) and all deceeding \(0\) to \(0\).
The
resulting image is shown in
Figure~\ref{fig:sailboat}\,(a). We then apply four
different first and second oder differences based approaches, where for each,
the best result among the range of parameter from
\(\alpha\coloneqq\alpha_1=\alpha_2 \in \frac{1}{32}\mathbb N_0\), i.e. the grid consisting of multiples of \(\frac{1}{32}\) including \(0\) as long as \(\beta\neq0\) and the same range 
applies for \(\beta\coloneqq\beta_1=\beta_2=\gamma\). We further for this example \(\alpha_3=\alpha_4=0\), i.e. we focus on the purely anisotropic first order discrete total variation and a second order isotropic version. We measure the
quality using a peak signal to noise ratio (PSNR) and the structural similarity index (SSIM) on RGB color space and. We record the best result with respect to PSNR in~\ref{fig:sailboat}\,(b)--(e) and the original is shown in \ref{fig:sailboat}\,(f) for comparison of the enlarged regions.

First, we apply a real valued approach to each of the RGB channels separately. For \(\alpha=\tfrac{1}{32}\),
\(\beta=\tfrac{1}{16}\) we obtain the best PSNR of \(21.41\), which
is shown in Figure~\ref{fig:sailboat}\,(b).
Applying a vector valued approach on RGB, i.e.\ setting \(m=0, n=3\) in
Algorithm~\ref{alg:CPPA}, the best result obtains a PSNR of \(21.46\) for \(\alpha=\beta=\tfrac{1}{16}\), cf. Figure~\ref{fig:sailboat}\,(c). While this outperforms the component wise denoising by taking combined color space edges into account, it does not reconstruct the colorfullness of the image.

On the other hand, we apply Algorithm~\ref{alg:CPPA} to each of the channels of
HSV, i.e., setting \(m=1,n=0\) for the first channel and taking the real valued
case from above for the second and third. In order to keep the channels
unscaled, the algorithm presented in this paper is rescaled to run w.r.t. to \(\bmod\,1\).
The result is shown in
Figure~\ref{fig:sailboat}\,(d) yielding a PSNR of
\(22.08\). Finally, applying a vector valued approach on HSV, i.e.\ setting
\(m=1\) and \(n=2\) ---again having the cyclic channel w.r.t.\ \(\bmod\,1\)--- 
yields an image shown in
Figure~\ref{fig:sailboat}\,(e), having a PSNR of \(23.43\), the best result of all four compared algorithms. Note that especially the colors are much better reconstructed than in the RGB based denoising approaches, which both suffer from reduced saturation. This increase in quality can also be seen in the SSIM values denoted at the subfigure captions in the bottom right. Interestingly the SSIM value of the componentwise approach on RGB is just a bit better than the vector valued approach. Furthermore, edges can be much better recognized in the vectorial approaches than in both channel-wise approaches, see especially the magnified region of the sail.
\begin{figure}[t]\centering
		\includegraphics[scale=.9]{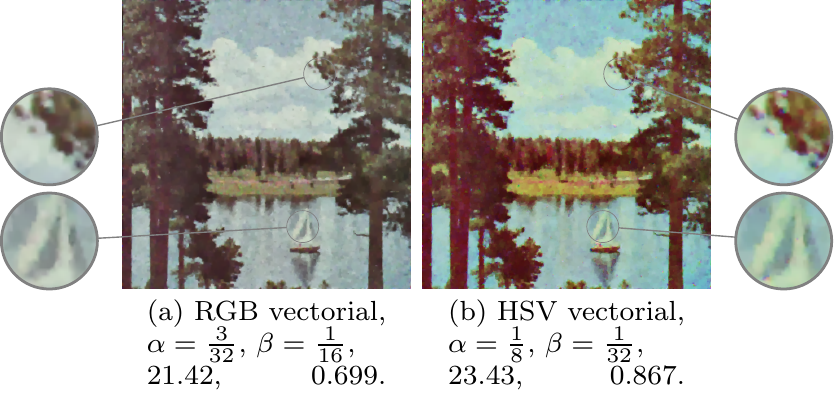}
	\caption{The best results on \(\frac{1}{32}\mathbb N_0\) of the parameters and the noise described in~\ref{fig:sailboat2} but with respect to SSIM. While the best result w.r.t.\ SSIM coincides with the best PSNR value for the vectorial HSV approach in, see\,(b), the best vectorial RGB approach only increases slightly in SSIM, still suffering the same effects.}\label{fig:sailboat2}
\end{figure}

Nevertheless there is are two tradeoffs. One concerns the overall amount of the parameters: Increasing the parameters smoothens the image and removes noise in larger constant regions like the sky but also destroys small details, like the leaves or the sail. The best PSNR is something in between: while the sky still resembles a little noise, the sail is only smoothened a little, but its main features are kept.
The second tradeoff is between the first and second order terms. While a dominant value of \(\alpha\) would keep edges in the image, it also introduces the well-known stair casing effect. This effect is reduced by the second order term, i.e.\ the parameter(s) \(\beta\), which also smoothens edges. Due to the existence of both smooth regions and edges in natural images, a certain equilibrium has to be chosen.
%
%
\subsection{Inpainting nonlinear color space data}\label{subsec:Inpainting}
\paragraph{Inpainting noise-free data.}\par
Here we consider the situation where some data items are missing, are lost or have been removed by a user. As example space we again consider the HSV color space.  
As in Section \ref{subsec:NLcolorspaceAppl}, 
we compare the results with the usual approach using the linear RGB color space and with the HSV approach working component-wise.

We consider a synthetic image in the HSV color space given by the function
\begin{equation}\label{eq:Formula4SynthImage}
\bigl(
\text{\lstinline!atan2!}\frac{x}{y},\ 1-x^2,\ 1-\abs{x+y}
\bigr),
\quad
x,y\in[-\tfrac{1}{2},\tfrac{1}{2}]^2,
\end{equation}
where the first component is the arctangent function with two arguments. 
This extends a synthetic $\mathbb S^1$ example used by the authors in \cite{BW15}.
The original image is shown in Figure~\ref{fig:synth}\,(f). The initial data is obtained by removing a disc with radius \(r=\frac{1}{4}\) as shown in
Figure~\ref{fig:synth}\,(a). The goal is to ``recover'' the image in Figure~\ref{fig:synth}\,(f).
For the inpainting we again apply Algorithm~\ref{alg:CPPA} using \(K=800\) iterations and performing a parameter search on \(\alpha\coloneqq\alpha_1=\alpha_2, \beta\coloneqq\beta_1=\beta=2=\gamma \in\frac{1}{8}\mathbb N_0\). Then, the real valued approaches in RGB color space, both channel-wise (b) and  vectorial (c), do not reconstruct the original colors correctly.
In contrast, the channel-wise approach (d) and the vectorial
approach (e) on the HSV space keep the colors and both produce very satisfactory results.
The result of the vectorial approach in Figure~\ref{fig:synth}\,(e) is a nuance ``sharper'', its reconstruction is a little bit better in PSNR at least.
\begin{figure*}[t]\centering
	\includegraphics[scale=.9]{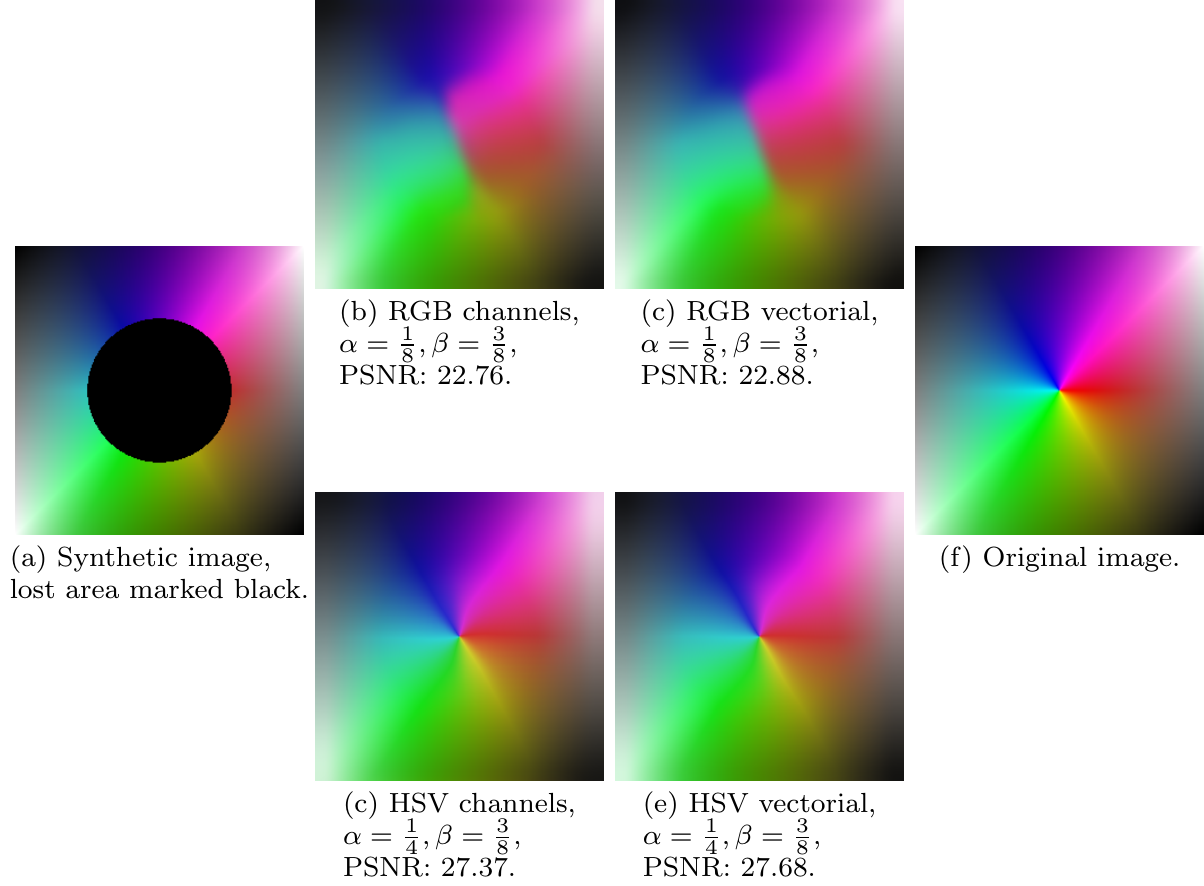}
	\caption{Reconstruction of a synthetic image with the black inner circle missing in a noiseless setup. The RGB-based reconstructions (a),(b) both yield a degrading of the colors, while the  HSV-based approaches (c),(d) reconstruct the colors. 
	The proposed approach 
	(d) based on a \(\mathbb S^1\times \mathbb R^2\) model yields the best PSNR.}
	\label{fig:synth}
\end{figure*}
\paragraph{Inpainting and denoising data.}\par
In many situations data is noisy and parts are lost or invalid. 
This results in an combined inpainting and denoising problem for which we apply the proposed methods next. As in Section~\ref{subsec:NLcolorspaceAppl} before, 
we consider the HSV color space and
compare the results with the usual approach using the linear RGB color space and with the HSV approach working component-wise. 

As a test scenario, we add wrapped Gaussian and Gaussian noise with \(\sigma = \frac{1}{5}\) to the cyclic and non-cyclic components, respectively, similar to Section~\ref{subsec:NLcolorspaceAppl}. Furthermore, we remove a ball in the center as also done in Section \ref{subsec:Inpainting}; see Figure~\ref{fig:synth_noise}.
This initial is shown in Figure~\ref{fig:synth_noise}\,(a). 
Again, we would like to get back the image shown in Figure~\ref{fig:synth_noise}\,(f), which is data defined in HSV space and therefore also has features e.g.\ in the hue. The parameters for the algorithm are obtained using the same setup for Algorithm~\ref{alg:CPPA} as in Section~\ref{subsec:NLcolorspaceAppl}. Then, the real valued approaches in RGB color space, both channel-wise (b) and  vectorial (c), do not reconstruct the significant features.
In contrast, the channel-wise approach (d) and the vectorial
approach (e) on the HSV space reconstruct all significant features and produce almost perfect results.
The result of the vectorial approach in Figures~\ref{fig:synth_noise}\,(e) is a nuance ``sharper'' yielding a slightly higher PSNR.

\begin{figure*}[t]\centering
	\includegraphics[scale=.9]{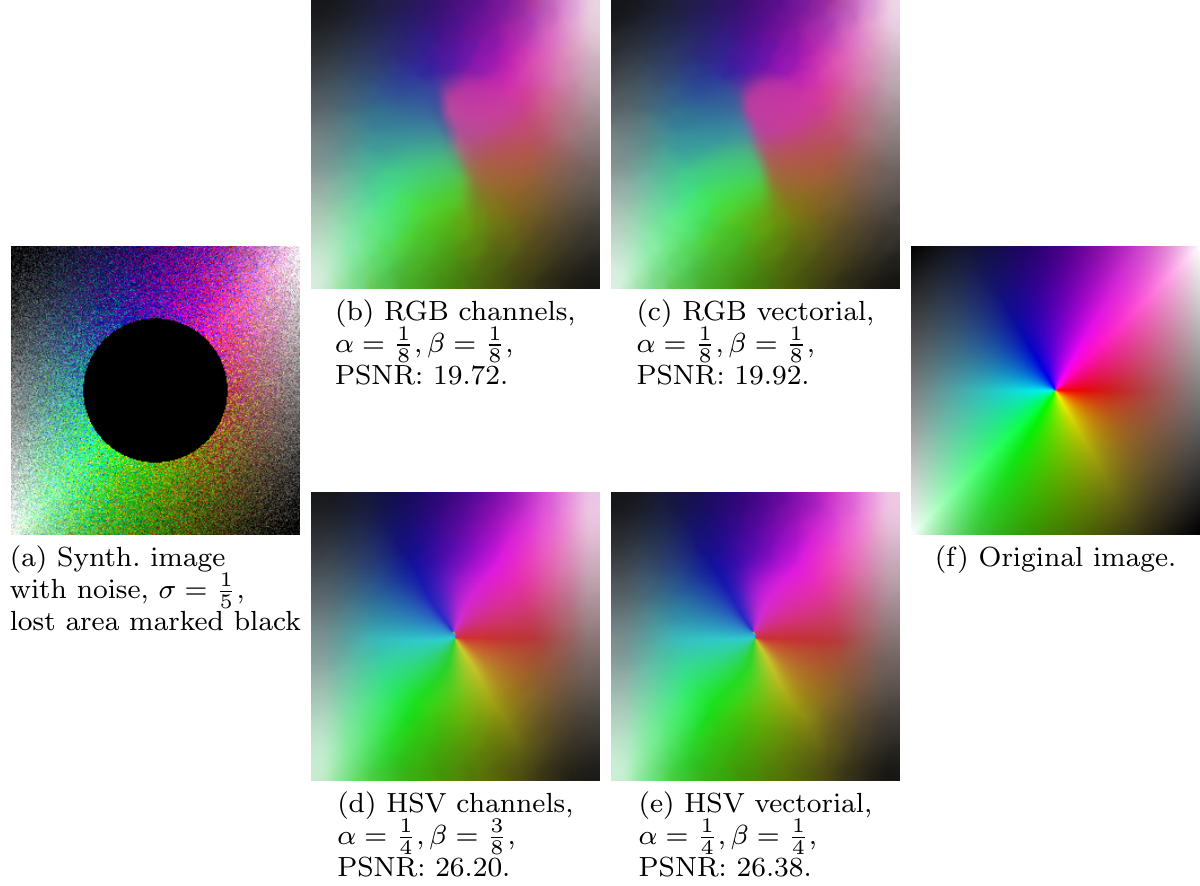}
	\caption{Reconstruction of a synthetic image with the black inner circle indicating missing data and the measurement itself being noisy. The RGB based reconstructions (b),(c) miss the main smooth features. The HSV model based reconstructions (d),(e) reconstructs these features yielding a satisfactory result. The proposed vectorial approach (e) yields the best PSNR.}
	\label{fig:synth_noise}
\end{figure*}
%
\paragraph{Inpainting real world data.}\par
Looking at real world data, we compare an inpainting of the image ``beach''\footnote{Available at \href{http://pixabay.com/de/strand-lagune-sonnenuntergang-164288/}{http://pixabay.com/de/strand-lagune-sonnenuntergang-164288/} and published under public domain.}.
In Figure~\ref{fig:beach}\,(\subref{subfig:beach:lossy}) some areas of the image were lost. We
compare the vectorial approach on RGB color space in
Fig.~\ref{fig:beach}\,(\subref{subfig:beach:rgb}) and HSV in Fig.~\ref{fig:beach}\,(\subref{subfig:beach:hsv}) for the same
set of parameters \(\alpha=\frac{1}{16}\), \(\beta=\frac{1}{32}\): the HSV color model is able to
also incorporate color changes as shown in the magnification, where the
HSV model is able to reconstruct the violet part of the clouds. We note that we employed the inpainting algorithm as described in Algorithm~\ref{alg:CPPA} also
for the diagonal differences. Even iterating alternatingly the anisotropic differences yields preference of the diagonal directions, which is even enhanced by the diagonal differences. Therefore both models tend to create wavy structures, when the surrounding is not isotropic as in the last examples. 

Despite the most bottom left inpainting area, the HSV model performs better than the RGB model with respect to PSNR, where the HSV model obtains \(30.612\) and the RGB model is slightly behind with \(30.548\). The same holds for the SSIM, though there the difference is quite small, because even the SSIM of the lost version ---setting lost pixel to white as in Figure~\ref{fig:beach}\,(\subref{subfig:beach:lossy})--- yields an SSIM value of \(0.9202\). We obtained an SSIM value of \(0.98866\) for the complete image with the HSV based inpainting and \(0.98862\) for the RGB based inpainting.

\begin{figure*}[tbp]\centering
	\begin{subfigure}{.230\textwidth}\centering
		\includegraphics[width=\textwidth]{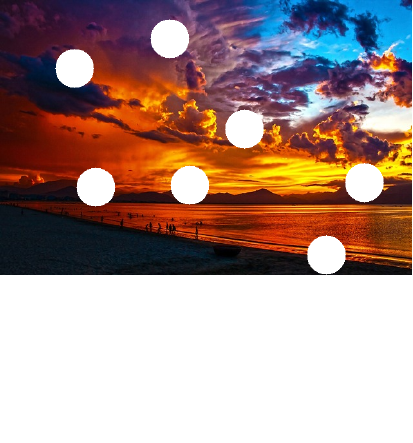}
		\caption{Lost area (white).}\label{subfig:beach:lossy}
	\end{subfigure}
	\begin{subfigure}{.230\textwidth}\centering
		\includegraphics[width=\textwidth]{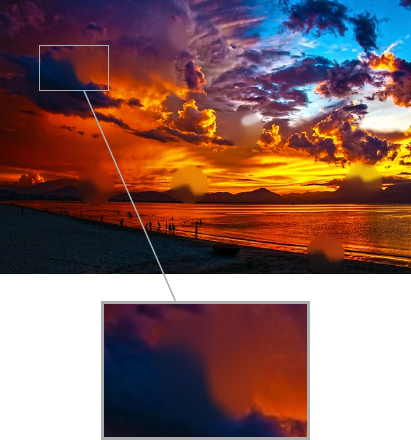}
		\caption{RGB model.}\label{subfig:beach:rgb}
	\end{subfigure}
	\begin{subfigure}{.230\textwidth}\centering
		\includegraphics[width=\textwidth]{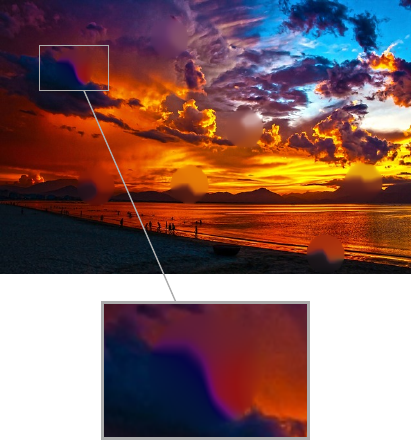}
		\caption{ HSV model.}\label{subfig:beach:hsv}
	\end{subfigure}
	\begin{subfigure}{.230\textwidth}\centering
		\includegraphics[width=\textwidth]{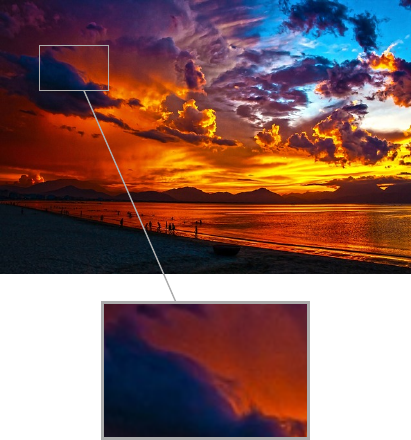}
		\caption{Original image.}\label{subfig:beach:orig}
	\end{subfigure}
	\caption{Inpainting the real world image ``beach'': The HSV based model better reconstructs the clouds than an RGB based approach.}\label{fig:beach}
\end{figure*}
%
\subsection{Denoising sections in volumetric cyclic data} 
\label{sec:denoisingVolumes}

Finally, we apply our algorithms for denoising frames in volumetric data.
Examples of volumetric data are frames of a 2D film or a stack of slices,
each slice being a 2D image as appearing, e.g., in computed tomography.
We want to denoise such slices incorporating the temporal/spatial information 
stemming from the third dimension, which is, e.g., the temporal neighborhood information 
in a film. 

To be more precise, we consider volumetric data $I_{k}(i,j),$ where 
$I_{k}(i,j)$ is a vector valued pixel at the pixel location $i,j$ in the $k$th
frame/slice.
The setup is rather general, and we can assume $I_{k}(i,j)$ being data from some rather general space $\mathcal M$ -- say a manifold; here we exemplarily consider $\mathcal M=\mathbb{S}^1.$
We take a look at the $l$ neighboring (left and right) frames
$I_{h}, h=k-l,\ldots,k,\ldots,k+l$, around a center frame $I=I_k$ at position $k.$  
With the position $k$, we now associate bivariate data living in $M^{2l+1}$ being the vector of data points at the same position \((i,j)\) in the  neighboring frames. To be precise, we consider the bivariate data $J_k$, with $J_k(i,j) \in M^{2l+1}$ given by 
\( J_k(i,j) = \Bigl( I_{k-h}(i,j)\Bigr)_{h=-l}^{l}.
\)
We apply our algorithms to the derived data $J_k$ and compare the result to the usual denoising of the single frame $I_k$.

The video underlying Figure~\ref{fig:video_vect} is constructed as follows.
As basis, we use the image given by the first component
of~\eqref{eq:Formula4SynthImage} on \([-\tfrac{1}{2},\tfrac{1}{2}]^2\); outside the disc of radius \(\tfrac{1}{2}\), we add \(\tfrac{\pi}{4}\) which is the same as rotating the input \((i,j)\) of each pixel in \(\Omega_0\) by the same amount clockwise. The video consists of \(13\) frames rotating the disc clockwise by \(\pi\) and the outer region by \(\pi\) counterclockwise, i.e.\ from \(-\tfrac{\pi}{2}\) before the center frame to \(\tfrac{\pi}{2}\) at the end of the sequence for the inner and with changed signs for the outer redgion. This corresponds to a rotation of \(\tfrac{\pi}{12}\) per frame and region. We finally sample each of these frames with \(256\times 256\) pixel on \([-\tfrac{1}{2},\tfrac{1}{2}]^2\).

In our example, Figure~\ref{fig:video_vect}, we show in (d) the seventh frame
of the constructed video from the previous the paragraph.
On each frame, we impose wrapped Gaussian noise with standard deviation \(\sigma = \frac{2}{5}\), see Figure~\ref{fig:video_vect}\,(a) for frame 7 of the video. In~(b), we perform denoising just on the frame \(k=7\). Performing a first and second order denoising yields staircasing and/or reduction of the sharp edge at the disc border. Choosing the parameters \(\alpha=\alpha_1=\alpha_2,\beta=\beta_1=\beta_2\) from the set \(\in\tfrac{1}{64}\mathbb N_0\) and setting \(K=400\) as maximal number of iterations of the CPPA, we obtain the optimal value \(\alpha=\tfrac{1}{32}\), \(\beta=0\). This indicates that staircasing still resembles a better result than unsharpening the edge, which would be the effect of including second order differences.
In (c), we perform a combined vectorial denoising in \((\mathbb S^1)^{13}\)
as proposed for a total of \(13\) frames $(l=6, k=7)$ and show the central seventh frame. 
At the cost of being computationally more expensive due to the increased data set, this approach outperforms the first, single frame based approach.
\begin{figure*}[tbp]\centering
	\includegraphics[scale=.9]{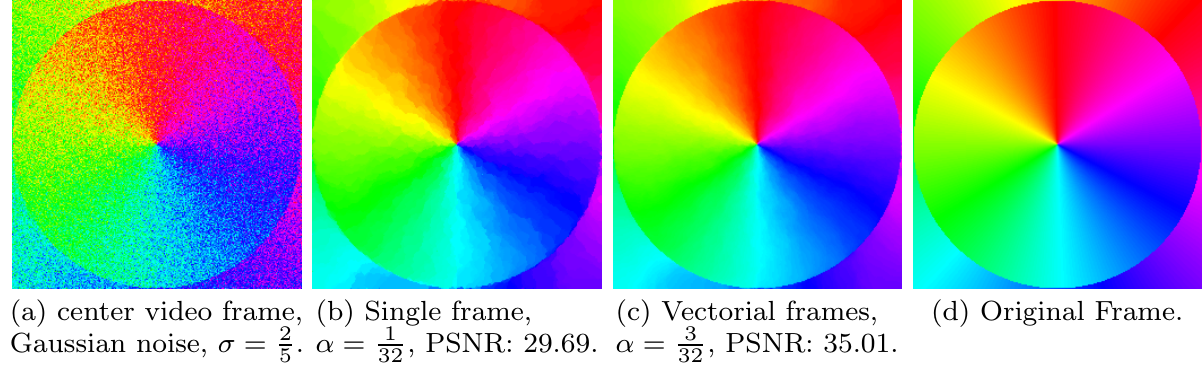}
	\caption{Denoising $\mathbb{S}^1$-valued video frames. The proposed approach for denoising $\mathbb{S}^1$-valued video frames (c) yields significantly better results than  frame-wise denoising (b).}
	\label{fig:video_vect}
\end{figure*}
\section*{Discussion and Future Research}

In the following we discuss the relation of the present work to the author's previous work~\cite{bergmann2014second,BW15}. These works consider the $\mathbb S^1$-valued situation, whereas we here consider the product space $(\mathbb S^1)^m \times \mathbb R^n.$ 
These product spaces are practically relevant: they appear, e.g., in the context of color space and, more abstractly, whenever considering polar coordinates (magnitude and phase).
In contrast to general manifolds the spaces considered here still allow for a solution to the considered minimization problem, i.e.\ the involved proximal mappings, to be given in explicit form.
We are convinced that this is no longer possible for general manifolds. The product spaces further imply changes in both the algorithms and the convergence analysis.
Although this paper as well as~\cite{bergmann2014second,BW15} (and also the previous work \cite{WDS2013}) use a CPPA, the derived proximal mappings are different. In particular the mappings derived here do not arise as component-wise application of the one dimensional situation. This is especially not the case for the setting, where several data items are fixed by constraint.
In this context, we again mention that the natural second differences on $(\mathbb S^1)^m \times \mathbb R^n$ which we defined in this paper couple the components of the range space, which also affects the convergence analysis
The employed methods are based on the ones used in~\cite{bergmann2014second} (which themselves are based on~\cite{WallnerDynCAGD}), but the concrete analysis involves an additional degree of complexity 
stemming from the mentioned coupling.

This work also contains material which is even new for $\mathbb S^1$-valued data.
First, in contrast to pure denoising, additional/different proximal mappings are needed when dealing with the inpainting problem due to the additional constraints.
In~\cite{BW15}, we simply used the proximal mappings computed
in~\cite{bergmann2014second} and
applied projections to ensure the constraints.
Here, we compute the proximal mappings of the constrained problems explicitly. 
This is more natural and also advantageous in the analysis.
Concerning the analysis, we here include an inpainting setup. 
Even for the $\mathbb S^1$ setting, this has not yet been done. 
In the conference proceeding~\cite{BW15} we do not provide a convergence analysis.
The paper~\cite{bergmann2014second} which provides an analysis considers functionals for denoising in the $\mathbb S^1$ setting only.

A topic of future research are algorithms for higher order TV-type functionals for data living in more general manifolds.

\paragraph{Acknowledgement.}
	This research was started when RB visited the 
	Helmholtz-Zentrum M\"unchen in summer 2014.
	We thank Gabriele Steidl for valuable discussions.
	AW is supported by the Helmholtz Association within the young
	investigator group VH-NG-526. AW also acknowledges the support
	by the DFG scientific network ``Mathematical Methods in Magnetic Particle Imaging''.
%
\renewcommand*{\bibfont}{\footnotesize}
\setlength\bibitemsep{0pt}
\printbibliography
%
%
\end{document}